\documentclass[10pt,reqno]{amsart}

\usepackage{amsmath, amsfonts, amssymb, amscd, enumerate, stmaryrd}
\usepackage[all]{xy}
\usepackage{hyperref}
\newtheorem{theo}{Theorem}[section]

\newtheorem{definition}[theo]{Definition}
\newenvironment{pf}{\noindent{\it Proof. }}{$\square$\par\medskip}

\newtheorem{lemma}[theo]{Lemma}
\newtheorem{theorem}[theo]{Theorem}
\newtheorem{corollary}[theo]{Corollary}
\newtheorem{proposition}[theo]{Proposition}

\newtheorem{remark}[theo]{Remark}

\newcommand{\beq}{\begin{equation}}
\newcommand{\eeq}{\end{equation}}

\renewcommand{\i}{\iota}

\newcommand{\q}{\mathrm{q}}

\newcommand{\p}{\mathrm{p}}

\newcommand{\bC}{\mathbb{C}}
\newcommand{\bR}{\mathbb{R}}
\newcommand{\bZ}{\mathbb{Z}}

\newcommand{\bH}{\mathbb{H}}
\newcommand{\bN}{\mathbb{N}}

\newcommand{\pc}{\phantom{c}}

\renewcommand{\gg}{\mathfrak{g}}
\newcommand{\gh}{\mathfrak{h}}

\newcommand{\gm}{\mathfrak{m}}

\newcommand{\so}{\mathfrak{so}}

\newcommand{\spin}{\mathfrak{spin}}

\newcommand{\ggl}{\mathfrak{gl}}
\newcommand{\gheis}{\mathfrak{heis}}

\newcommand\GL{\mathrm{GL}}

\newcommand\Cl{\mathrm{Cl}}
\newcommand\mrp{\mathrm{p}}

\newcommand\T{\mathrm{T}}

\newcommand{\cC}{\mathcal{C}}

\renewcommand{\square}{\kern1pt\vbox
{\hrule height 0.6pt\hbox{\vrule width 0.6pt\hskip 3pt
\vbox{\vskip 6pt}\hskip 3pt\vrule width 0.6pt}\hrule height0.6pt}\kern1pt}

\DeclareMathOperator\ad{ad}

\DeclareMathOperator\End{End\;}
\DeclareMathOperator\Ker{Ker\;}
\DeclareMathOperator\Hom{Hom\;}

\DeclareMathOperator\Bil{Bil\;}

\DeclareMathOperator\Id{Id}

\DeclareMathOperator{\Der}{Der\;}

\DeclareMathOperator{\CoKer}{Coker}
\DeclareMathOperator{\diag}{diag}
\renewcommand\Re{\operatorname{Re}}
\renewcommand\Im{\operatorname{Im}}

\newcommand{\0}{\overline{0}}
\newcommand{\ou}{\overline{1}}
\newcommand{\be}{\begin{equation}}
\newcommand{\ee}{\end{equation}}

\def\<#1,#2>{\langle\,#1,\,#2\,\rangle}
\newcommand{\arr}{\begin{array}{rlll}}
\newcommand{\ea}{\end{array}}
\newcommand{\bea}{\begin{eqnarray}}
\newcommand{\eea}{\end{eqnarray}}
\newcommand{\bean}{\begin{eqnarray*}}
\newcommand{\eean}{\end{eqnarray*}}

\def\sideremark#1{\ifvmode\leavevmode\fi\vadjust{
\vbox to0pt{\hbox to 0pt{\hskip\hsize\hskip1em
\vbox{\hsize3cm\tiny\raggedright\pretolerance10000
\noindent #1\hfill}\hss}\vbox to8pt{\vfil}\vss}}}

\newcounter{ssig}
\newcounter{ttig}
\begin{document}
\title[Superizations of Cahen-Wallach symmetric spaces]{Superizations of \\ Cahen-Wallach symmetric spaces\\  and\\  spin representations of \\ the Heisenberg algebra}
\author[Andrea Santi]{Andrea \ Santi}

%
\maketitle
\setcounter{equation}{0}
\bigskip
{\bf Abstract:} Let $M_{0}=G_{0}/H$ be a $(n+1)$-dimensional Cahen-Wallach Lorentzian symmetric space associated with a symmetric decomposition $\gg_{0}=\gh+\gm$ and let $S(M_{0})$ be the spin bundle defined by the spin representation $\rho:H\rightarrow\GL_{\bR}(S)$ of the stabilizer $H$. 
This article studies the superizations of $M_{0}$, \textit{i.e.} its extensions to a homogeneous supermanifold $M=G/H$ whose sheaf of superfunctions is isomorphic to $\Lambda(S^{*}(M_{0}))$. Here $G$ is a Lie supergroup which is the superization of the Lie group $G_{0}$ associated with a certain extension of the Lie algebra $\gg_{0}$ to a Lie superalgebra $\gg=\gg_{\0}+\gg_{\ou}=\gg_{0}+S$, via the Kostant construction. The construction of the superization $\gg$ consists of two steps: extending the spin representation $\rho:\gh\rightarrow\ggl_{\bR}(S)$ to a representation $\rho:\gg_{0}\rightarrow\ggl_{\bR}(S)$ and constructing appropriate $\rho(\gg_{0})$-equivariant bilinear maps on $S$. Since the Heisenberg algebra $\gheis$ is a codimension one ideal of the Cahen-Wallach Lie algebra $\gg_{0}$, first we describe spin representations of $\gheis$ and then determine their extensions to $\gg_{0}$. There are two large classes of spin representations of $\gheis$ and $\gg_{0}$: the zero charge and the non-zero charge ones. The description strongly depends on the dimension $n+1$ (mod $8$). Some general results about superizations $\gg=\gg_{\0}+\gg_{\ou}$ are stated and examples are constructed.
\tableofcontents
\section*{Introduction}
Lie superalgebras have played an important role in modern physics since the idea of supersymmetry arose. Complex and real simple Lie superalgebras were classified by Kac (see \cite{K, Sc}). This classification was used in \cite{N} to describe superizations (\textit{i.e.} extensions of a Lie algebra $\gg_{0}$ to a Lie superalgebra $\gg=\gg_{\0}+\gg_{\ou}=\gg_{0}+\gg_{\ou}$) of different fundamental Lie algebras of symmetries which appear in physics. The work of Nahm was important for the construction of various theories of supergravity (see \cite{CJS}).
In this spirit, the classification of superizations of Poincar\'e Lie algebras, in all signatures and dimensions, has been achieved (see \cite{AC, ACDP, DFLV}). 
\\
Lorentzian symmetric spaces $M_{0}=G_{0}/H$ described by Cahen-Wallach (see \cite{CW}) appear in constructions of maximally supersymmetric solutions of $11$-dimensional supergravity (see \cite{FP, FGH}). In this context, some special, physically relevant superization $\gg=\gg_{\0}+\gg_{\ou}$ of the corresponding Cahen-Wallach Lie algebra $\gg_{0}=\gh+\gm$, $\dim\gm=11$, has been considered: the action of the even part $\gg_{\0}$ on the odd part $\gg_{\ou}=S$ is a spin representation, \textit{i.e.} an extension $\rho:\gg_{0}\rightarrow\ggl_{\bR}(S)$ of the spin representation $\rho:\gh\rightarrow\ggl_{\bR}(S)$ of the stability subalgebra $\gh$. The Cahen-Wallach Lie algebra has the form
$$
\gg_{0}=Lie(G_{0})=\gg_{b}=\gh+\gm=E^{*}+(E+\bR\p+\bR\q)
$$
where $E+\bR\p+\bR\q$ is the decomposition of Minkowski space in direct sum of Euclidean space $E$ and $2$-dimensional Minkowski space with isotropic basis $\p$, $\q$. The Lie bracket is determined by a non-degenerate symmetric bilinear form $b$ on the Euclidean space $E$.
\\

This paper studies superizations $\gg=\gg_{\0}+\gg_{\ou}$ of the Cahen-Wallach Lie algebra, where the adjoint action of the even part $\gg_{\0}$ on the odd part $\gg_{\ou}=S$ is given by a spin representation. All dimensions are considered but, due to Bott-periodicity in Clifford theory, the results mainly depend on $\dim\gm=n+1$ mod $8$.
\\
The Cahen-Wallach Lie algebra $\gg_{0}=\gheis+\bR\q$ is a one-dimensional extension of the Heisenberg algebra $\gheis=E^{*}+E+\bR\p$ determined by an outer derivation $\ad_{\q}\in\Der_{\bR}(\gheis)$. As an important intermediate step all spin representations $\rho:\gheis\rightarrow\ggl_{\bR}(S)$ of the Heisenberg algebra are classified. There are two natural classes of such representations, which depend on the image $\rho(\p)$ (called the \textit{charge}) of the central element $\p\in\gheis$: zero charge representations and non-zero charge representations. Their description depends on the solution of some quadratic equation in the even Schur algebra (which is isomorphic to $\bR$, $\bC$ or $\bH$) of the $\Cl(E)$-module $S_{0,n-1}$. Zero charge representations appear in all dimensions and are described in a unified way in Theorem \ref{0c}; they correspond bijectively to suitable pairs (see Definition \ref{suitable}). Zero charge representations are the only spin representations of the Heisenberg algebra when semi-spinors do not exist, as shown in Theorem \ref{al}. In the case when semi-spinors exist, there are non-zero charge spin representations which are described in Theorems \ref{0p}, \ref{3p}, \ref{2p}, \ref{3d} in terms of suitable maps. In the case $\dim E=8$, using representation theory of semisimple Lie algebras, Theorem \ref{aaa} describe all suitable maps and we get parametrization of such representations.
\\
Note that any spin representation of the Heisenberg algebra $\gheis$ defines an odd-commutative superization of $\gheis$, \textit{i.e.} a superization with the the trivial odd bracket $[S,S]=0$. In the case $\dim E=8$, Theorem \ref{aaa} shows that non-zero charge spin representations can be extended only to odd-commutative superizations of the Heisenberg algebra. In the case of zero charge, Proposition \ref{zzz} gives a description of a class of non odd-commutative superizations and Proposition \ref{zzzz} gives explicit examples. Non trivial superizations of the Heisenberg algebra have been recently used in theoretical physics (see \cite{CKN}). 
\\
Section \ref{dan} considers the extension of spin representations of the Heisenberg algebra to the Cahen-Wallach Lie algebra. The problem reduces to determining the image $\rho(\q)$ which satisfies appropriate commutative relations. Theorem \ref{iiii} shows that zero charge spin representations of the Cahen-Wallach Lie algebra are determined by solutions of a fundamental quadratic equation in the Clifford algebra $\Cl(E)$. The problem of extending non-zero charge spin representations of the Heisenberg algebra is more complicated and depends on $n+1$ mod $8$. We specialize to the case $\dim E=8$ and prove that non-zero charge spin representations exist if and only if the bilinear form $b$ is proportional to Euclidean metric (see Theorem \ref{check}). An example of such representation is given and it is checked that the obtained formula gives a representation in any dimension when semi-spinors exist. 
\\
Section \ref{dan2} considers superizations of the Cahen-Wallach Lie algebra. Theorem \ref{zzrr} gives a characterization of all zero charge superizations with translational supersymmetry. It would be interesting to construct examples of non odd-commutative superizations with non-zero charge.
\\

There is a geometric interpretation of the procedure of superization in the framework of supergeometry. The Spin bundle $S(M_{0})$ of a Lorentzian spin manifold $M_{0}$ defines a supermanifold $M=(M_{0},\mathcal{A}_{M})$ whose sheaf of superfunctions $\mathcal{A}_{M}$ is isomorphic to the sheaf of sections of the exterior algebra $\Lambda(S^{*}(M_{0}))$ of $S^{*}(M_{0})$. Supermanifolds of this type are studied in \cite{ACDS, Kl1, Kl2}. Moreover if $M_{0}=G_{0}/H$ is a homogeneous Lorentzian manifold and $S(M_{0})$ is the spin bundle defined by the spin representation $\rho:H\rightarrow\GL_{\bR}(S)$ of the stabilizer $H$ then a superization $\gg=\gg_{\0}+\gg_{\ou}=\gg_{0}+S$ of the Lie algebra $\gg_{0}=Lie(G_{0})=\gh+\gm$ defines a structure of homogeneous supermanifold $G/H$ on $M$ (see \cite{S}). Here $G$ is the Lie supergroup associated with the Harish-Chandra pair $(G_{0},\gg)$, via the Kostant construction (see \cite{Kt, Kz}).
\section{Preliminaries} \label{Cahen-Wallach}
\setcounter{equation}{0}
\bigskip
\par
\subsection{Cahen-Wallach algebra}\hspace{0 cm}\newline\\
Let $(\gm,\left\langle\cdot ,\cdot\right\rangle)$ be a $(n+1)$-dimensional Minkowski space ($n\geq 2$), \textit{i.e.} a real vector space endowed with an inner product $\left\langle\cdot ,\cdot\right\rangle $ of signature $(1,n)=(+,-)$ and fix a Witt decomposition
\bigskip
$$\gm=E\oplus\bR\mathrm{p} \oplus\bR\mathrm{q}$$
\\
with negatively defined scalar product $\left\langle\cdot ,\cdot\right\rangle_{|E}$, $\left\langle \mathrm{p},E\right\rangle=\left\langle \mathrm{q},E\right\rangle=\left\langle \mathrm{p},\mathrm{p}\right\rangle=\left\langle \mathrm{q},\mathrm{q}\right\rangle=0$ and $\left\langle \mathrm{p},\mathrm{q}\right\rangle=1$. Let $E^{*}$ be the dual space of $E$. Denote the musical isomorphisms by
$$
\flat:E\longrightarrow E^{*}\qquad\qquad,\qquad\qquad\sharp:E^{*}\rightarrow E
$$
$$
\phantom{cccc}e\longrightarrow e^{\flat}:=\left\langle e,\cdot\right\rangle_{|E}\qquad\phantom{cccccccccc} e^{\flat}\rightarrow e\quad.
$$
\begin{definition}
\cite{CW}
\rm{Let $b$ be a non-degenerate symmetric bilinear form on the Euclidean subspace $E$ of the Minkowski space $\gm$. We call {\bf CW (Lie) algebra} associated with $b$ the Lie algebra with symmetric decomposition 
\be
\label{marino}
\gg_{b} = \gh + \gm=E^{*}+\gm=E^{*}+(E+\bR\mathrm{p}+\bR\mathrm{q})
\ee
where the only non trivial Lie brackets are 
$$ 
[\mathrm{q},e]=e^{\flat},
$$
$$
\phantom{ccccccc}[e^{\flat},f]=b(e,f)\mathrm{p},
$$
$$
\phantom{cccc}[e^{\flat},\mathrm{q}]=B(e)
$$
where $e,f\in E$ and $B\in\End_{\bR}(E)$ is defined by $b(\cdot,\cdot):=-\left\langle B\cdot,\cdot\right\rangle_{|E}$. 
}\end{definition}
Denote by $\gheis_{b}:=E^{*}+E+\bR\mathrm{p}$ the ideal which is isomorphic to the Heisenberg algebra $\gheis_{b}$, where the "Planck constants" are given by the eigenvalues of $B$. Therefore the CW algebra $\gg_{b}$ can be thought as an extension of $\gheis_{b}$ by means of the outer derivation $\ad_{\mathrm{q}}\in\Der_{\bR}(\gheis_{b})$. Note that
\begin{itemize}
\item[$\cdot$]
$\mathrm{Z}(\gg_{b})=\bR\mathrm{p}=(\gg_{b})^{''}$,
\item[$\cdot$]
$E^{*}$ and $E$ are Abelian subalgebras,
\item[$\cdot$]
$[\gm,\gm]=\gh$,
\item[$\cdot$]
$\gh$ contains no non trivial ideal of $\gg_{b}\phantom{c}.$ 
\end{itemize} 
The first property shows that $\gg_{b}$ is a solvable Lie algebra and the non-degeneracy of $b$ implies that $\gg_{b}$ is indecomposable (for the definition, see \cite{CW}). 
\begin{lemma}
\cite{CW}
\label{CWiso}
Two CW algebras $\gg_{b}$, $\gg_{b^{'}}$ are isomorphic if and only if there exist an orthogonal map $L\in\mathrm{O}(E)$ and a real $r\neq 0$ such that $b^{'}=r^{2}L^{*}b$.
\end{lemma}
The isomorphism is explicitly given by the following map
$$
\varphi:\gg_{b}\longrightarrow\gg_{b^{'}}
$$
$$
\phantom{ccc}\mathrm{p}\longrightarrow r\mathrm{p}
$$
$$
\phantom{cccc}\mathrm{q}\rightarrow r^{-1}\mathrm{q}
$$
$$
\phantom{cccc}e\rightarrow L^{-1}e
$$
$$
\phantom{ccccccccccccccccc}e^{\flat}\rightarrow r^{-1}(L^{-1}e)^{\flat}\phantom{cccccc}.
$$
We fix an orthonormal $B$-eigenbasis $\left\{e_{1},...,e_{n-1}\right\}$ of $E$ such that 
\be
\label{prima}
B=\diag(b_{i}) \phantom{ccc},\phantom{ccc}  b_{i}\leq b_{j} \phantom{ccc},\phantom{ccc}\ \sum b_{i}^{2}=n-1
\ee
where $b_{i}\neq 0$ are real numbers for $1\leq i\leq j\leq n-1$. The last assumption in (\ref{prima}) is not restrictive due to Lemma \ref{CWiso}. We then identify the Euclidean space $E$ with $\bR^{0,n-1}$ where $\bR^{r,s}$ is the pseudo-Euclidean vector space  of signature $(r,s)=(+,-)$. For every CW algebra (\ref{marino}), denote by
$$
\mathfrak{i_{b}}:E^{*}\hookrightarrow\so(\gm)
$$
\be
\label{ortogonale}
\phantom{cccc}e^{\flat}\longrightarrow \mathrm{p}\wedge Be
\ee
the isotropy representation of $E^{*}$, given by the adjoint action of $E^{*}$ on $\gm$ in $\gg_{b}$ where
$\mathrm{p}\wedge e:=\left\langle \mrp,\cdot\right\rangle e-\left\langle e,\cdot\right\rangle\mrp \in\so(\gm)$.
Note that $\p\wedge E$ is an abelian subalgebra of the full orthogonal Lie algebra of $\gm$ given by $\so(\gm)=\p\wedge E+\q\wedge E+\bR\p\wedge\q+\so(E)$.
\subsection{Clifford theory}\hfill\newline\\
Denote by 
$$
\Cl_{r,s}=\Cl(\bR^{r,s}):=\T(\bR^{r,s})/\left\langle v\otimes v+\left\langle v,v\right\rangle \right\rangle
$$
the Clifford algebra of the space $\bR^{r,s}$, \textit{i.e.} the quotient of the tensor algebra $\T(\bR^{r,s})$
by the ideal generated by all elements of the form $v\otimes v+\left\langle v,v\right\rangle $ ($v\in\bR^{r,s}$). It is a $\bZ_{2}$-graded associative real algebra $\Cl_{r,s}=\Cl_{r,s}^{\circ}+\Cl_{r,s}^{1}$ and we denote by $\alpha$ its parity automorphism 
$$
\alpha|_{\Cl_{r,s}^{\epsilon}}=(-1)^{\epsilon}\Id\qquad\qquad \epsilon\in\left\{0,1\right\}\qquad\qquad\qquad\qquad\quad.
$$
There exist natural embeddings $\bR^{r,s}\subseteq\Cl_{r,s}$ and
$$
\so(\bR^{r,s})=\so(r,s)\hookrightarrow\Cl_{r,s}\qquad,\qquad v\wedge w\mapsto\frac{1}{4}[v,w]\qquad.
$$
Denote by
\be
\label{ir}
\rho_{r,s}:\Cl_{r,s}\rightarrow\End_{\bR}(S_{r,s})
\ee
the real spin representation, \textit{i.e.} $S_{r,s}$ is a real irreducible $\Cl_{r,s}$-module.
It is known (see \cite{LM}) that every Clifford algebra admits at most two inequivalent irreducible real representations. These representations are equivalent when restricted to $\so(r,s)$. If the $\so(r,s)$-module $S_{r,s}$ is reducible, it decomposes into a direct sum
\be
\label{semi}
S_{r,s}=S_{r,s}^{-}+S_{r,s}^{+}
\ee 
of irreducible $\so(r,s)$-modules which are called semi-spinors. As for the notation, for every $c\in\Cl_{r,s}$, the Clifford action $\rho_{r,s}(c)$ is denoted by $c\cdot$ or, sometimes, by $c$, where the dot is omitted when the action is clear from the context.
\begin{definition}\cite{AC}
\label{tsi}\rm{
A $\so(r,s)$-invariant bilinear form $\Gamma\in\Bil_{\bR}(S_{r,s})^{\so(r,s)}$ on the spin module $S_{r,s}$ is called {\bf admissible} if it has the following properties:
\begin{itemize}
\item[1)] Clifford multiplication $v\cdot$ is either $\Gamma$-symmetric ($\tau=+1$) or \\$\Gamma$-skewsymmetric ($\tau=-1$),
\item[2)] $\Gamma$ is symmetric ($\sigma=+1$) or skew-symmetric ($\sigma=-1$), 
\item[3)] If the $\so(r,s)$-module $S_{r,s}$ is reducible,
then $S_{r,s}^{\pm}$ are either mutually orthogonal ($\i=+1$) or isotropic ($\i=+1$).
\end{itemize}
The three invariants $\tau,\sigma,\i\in\left\{+1,-1\right\}$ defined above are called type, symmetry and isotropy of the admissible bilinear form $\Gamma$. We denote by $\Bil_{\bR}(S)^{\tau\sigma=-1}$ the space of admissible bilinear forms with $\tau\sigma=-1$.
}\end{definition} 
\cite{AC} proves that it is possible to canonically choose an admissible $h\in\Bil_{\bR}(S_{r,s})^{\so(r,s)}$ and that $\Bil_{\bR}(S_{r,s})^{\so(r,s)}$ has a basis of admissible bilinear forms.
\begin{definition}\cite{AC}\rm{
A $\so(r,s)$-invariant endomorphism $C\in\End_{\bR}(S_{r,s})^{\so(r,s)}$ of the spin module $S_{r,s}$ is called {\bf admissible} if it has the following properties:
\begin{itemize}
\item[1)] Clifford multiplication $v\cdot$ either commutes ($\tau=+1$) or \\anticommutes ($\tau=+1$) with $C$,
\item[2)] $C$ is $h$-symmetric ($\sigma=+1$) or $h$-skew-symmetric ($\sigma=+1$),
\item[3)] If the $\so(r,s)$-module $S_{r,s}$ is reducible,
then either $CS^{\pm}\subseteq S^{\pm}$ ($\i=+1$) or $CS^{\pm}\subseteq S^{\mp}$ ($\i=+1$).
\end{itemize}
The three invariants $\tau,\sigma,\i\in\left\{+1,-1\right\}$ defined above are called type, symmetry and isotropy of the admissible endomorphism $C$. 
}\end{definition}
The space $$\cC_{r,s}:=\End_{\bR}(S_{r,s})^{\so(r,s)}$$
is called {\bf Schur algebra} and has a basis of admissible morphisms (see \cite{AC}). The parity automorphism
$$
\overline{\phantom{c}}:\cC_{r,s}\rightarrow\cC_{r,s}
$$
\be
\label{conf}
\qquad\qquad\qquad\pc C\rightarrow \overline{C}:=\tau(C)\cdot C
\ee
is well-defined. Its $+1$-eigenspace, \textit{i.e.} the linear subspace of endomorphisms with invariant $\tau$ equal to $+1$, is denoted by
$$
\cC_{r,s}^{\circ}:=\left\{C\in C_{r,s}|\tau(C)=+1\right\}
$$
and called the {\bf even Schur algebra}. It coincides with the algebra of endomorphisms compatible with the irreducible representation (\ref{ir}). It follows that $\cC_{r,s}^{\circ}$ is a real division algebra, isomorphic then to $\bR$, $\bC$ or $\bH$. The following lemma is (implicitly) used quite frequently.
\begin{lemma}\cite{LM}
\label{VOLUME}
The volume form $\omega_{0,m}=e_{1}\cdot\cdot\cdot e_{m}\in\Cl_{0,m}$ belongs to the center, \text{i.e} $\tau(\omega_{0,m})=1$  (resp. twisted center, \textit{i.e} $\tau(\omega_{0,m})=-1$) of the Clifford algebra $\Cl_{0,m}$ if $m$ is odd (resp. even). Then, in both cases, it commutes with the even part $\Cl_{0,m}^{\circ}$ of $\Cl_{0,m}$. Moreover it satisfies
$$
\omega_{0,m}^{2}=
\left\{\begin{matrix} (-1)^{m\phantom{+1}}\quad\rm{if}\qquad m\equiv 3,4\quad (mod\phantom{c}4)\\
(-1)^{m+1}\quad\rm{if}\qquad m\equiv 1,2\quad(mod\phantom{c}4)
\end{matrix}\right.
$$  
\end{lemma}
\subsection{Extending the Cahen-Wallach algebra}\hfill\newline\\
In the following definition $\gg_{0}$ is a Lie algebra which can be either the Heisenberg algebra
$$\gheis_{b}=E^{*}+E+\bR\mathrm{p}$$
or the CW algebra
$$\gg_{b}=\gh+\gm=E^{*}+(E+\bR\p+\bR\q)$$
and 
\be
\label{rapprspin}
\rho_{\spin}:\Cl(\gm)\longrightarrow\End_{\bR}(S)
\ee
is the real spin representation $S=S_{1,n}$ of the Clifford algebra $\Cl(\gm)=\Cl_{1,n}$ of the Minkowski space $\gm$.
\begin{definition}
\label{SE}
\rm{
A representation $\rho:\gg_{0}\rightarrow\ggl_{\bR}(S)$ is called {\bf spin} if 
\begin{itemize}
\item[$i)$]
$\rho_{|E^{*}}=\rho_{\spin}\circ\mathfrak{i_{b}}$
\end{itemize}
where $\i_{b}$ is given by (\ref{ortogonale}). A spin representation $\rho:\gg_{0}\rightarrow\ggl_{\bR}(S)$ has {\bf zero charge} if $\rho(\p)=0$. A Lie superalgebra
($\gg=\gg_{\0}+\gg_{\ou},[\cdot,\cdot])$ is called a {\bf superization} of $\gg_{0}$ if 
\begin{itemize}
\item[$i)$]
$\gg_{\0}=\gg_{0}$,
\item[$ii)$]
$\gg_{\ou}=S$,
\item[$iii)$] The adjoint action of $\gg_{\0}$ on $\gg_{\ou}$ is a spin representation.
\end{itemize} 
A superization $\gg=\gg_{\0}+\gg_{\ou}$ is called {\bf odd-commutative} if $[\gg_{\ou},\gg_{\ou}]=0$. A superization of a CW algebra is said to have {\bf translational supersymmetry} if $\left[S,S\right]\subseteq\gm$. 
}\end{definition}
For technical reasons, in the case of zero charge spin representation of the CW algebra, the condition $\rho(\gm)\subseteq\rho_{\spin}(\Cl(\gm))$ is implicitly assumed. Relation
$$
\rho_{\spin}(\mathrm{p}\wedge e)=\frac{1}{4}\rho_{\spin}([\mathrm{p},e])=\frac{1}{2}\rho_{\spin}(\mathrm{p}e)
$$
implies that
$$
\rho_{\spin}(\mathrm{p}\wedge e)\circ\rho_{\spin}(\mathrm{p}\wedge f)=\frac{1}{4}\rho_{\spin}(\mathrm{p}e\mathrm{p}f)=-\frac{1}{4}\rho_{\spin}(\mathrm{p}\mathrm{p}ef)=0\quad.
$$
In particular every spin representation restricted to $E^{*}$ is two-step nilpotent.
\subsection{Spin representation in Lorentzian signature} \label{SpinLor}
\subsubsection{Model of $\Cl(\gm)$-module}\hspace{0 cm}\newline\\
The irreducible $\Cl_{1,n}$-module (\ref{rapprspin}) can be described in terms of the irreducible $\Cl_{0,n-1}$-module $S_{0,n-1}$. The fixed Witt decomposition
$$
\gm=\bR^{1,n}=\bR^{1,1}\oplus\bR^{0,n-1}=(\bR\p+\bR\q)+E
$$
induces a $\bZ_{2}$-graded isomorphism of $\bZ_{2}$-graded algebras (see \cite{AC})
$$
\Cl_{1,n}\cong\Cl_{1,1}\hat{\otimes}\Cl_{0,n-1}\cong\bR(2)\hat{\otimes}\Cl_{0,n-1}
$$
where $\bR(2)$ is the algebra of real $2\times 2$ matrices and $\hat{\otimes}$ stands for the graded tensor product of $\bZ_{2}$-graded algebras. This isomorphism is defined on generators as
$$\Cl_{1,n}\supseteq\bR^{1,n}=\bR^{1,1}\oplus\bR^{0,n-1}\ni v_{1,1}+v_{0,n-1}\longrightarrow v_{1,1}\otimes 1 + 1\otimes v_{0,n-1}\in\Cl_{1,1}\hat{\otimes}\Cl_{0,n-1}. 
$$
The decomposition of the irreducible $\Cl_{1,1}$-module $S_{1,1}=\bR^{2}=S_{1,1}^{-}+S_{1,1}^{+}$ into semi-spinors induces a decomposition of $S=S_{1,n}$ given by
\be
\label{dec}
S= S_{1,1}\hat{\otimes}S_{0,n-1}=S_{1,1}^{-}\otimes S_{0,n-1}+ S_{1,1}^{+}\otimes S_{0,n-1}=:S_{-}+S_{+}
\ee
where $S_{\mp}$ is linearly isomorphic to $S_{0,n-1}$.
The last equalities of (\ref{dec}), in contrast with the first one, are of vector spaces and {\bf not} of $\Cl_{1,n}$-modules, more precisely, (\ref{dec}) is {\bf not} the decomposition of $S$ into semi-spinors. To indicate this, we use low indices. Write an element $Q=Q_{-}+Q_{+}$ of (\ref{dec}) as the column  
$$ 
Q=\left(
\begin{array}{c}
Q_{-} \\
Q_{+} \\
\end{array}
\right)
$$
and use matrix notation for endomorphisms. If we decompose $\bR^{1,1}=\bR\mathrm{p}\oplus\bR\mathrm{q}=N\oplus N^{*}$ with
$\mathrm{q}(\mathrm{p}):=2\left\langle \mathrm{q},\mathrm{p}\right\rangle=2$, then the spin module $S_{1,1}$ is identified with the exterior algebra 
$$S_{1,1}=\Lambda N=\bR\mathrm{p}\oplus\bR 1=S_{1,1}^{-}+S_{1,1}^{+}$$
and the action of the Clifford algebra $\Cl_{1,1}\cong\bR(2)$ is given in terms of exterior multiplication $\varepsilon(\p)$ and interior multiplication $\i(\q)$ as follows
$$\rho_{1,1}:\Cl_{1,1}\rightarrow \End_{\bR}(S_{1,1})\pc\pc\pc\pc,\pc\pc\pc\pc\mathrm{p}\mapsto \varepsilon(\p)\pc\pc\pc\pc,\pc\pc\pc\pc\mathrm{q}\mapsto -\i(\q)\quad.$$
\subsubsection{Restriction of the spin representation $\rho:\Cl(\gm)\rightarrow\End_{\bR}(S)$ to the commutative subalgebra $\i_{b}(E^{*})\subseteq\so(\gm)\subseteq \Cl(\gm)$}
\begin{proposition}
\label{pallino}
With respect to decomposition (\ref{dec}), the image under the representation (\ref{rapprspin}) of $\mathrm{p}\wedge e\in\mathrm{p}\wedge E\subseteq\so(\gm)$ is given by
$$
\rho_{\spin}(\mathrm{p}\wedge e)=\begin{pmatrix} 0 & \sqrt{2} e \\ 0 & 0 \end{pmatrix}
$$
where $e\in E$ acts on $S_{0,n-1}$ by Clifford multiplication of a vector with a spinor. 
\end{proposition}
\begin{pf}
With respect to the basis $\left\{\mathrm{p}^{'}:=	\frac{\mathrm{p}}{2\sqrt 2}, 1^{'}:=1\right\}$ of $S_{1,1}$, we have 
$$
\rho_{1,1}(\mathrm{p})=\begin{pmatrix} 0 & 2\sqrt 2 \\ 0 & 0 \end{pmatrix}\qquad,\qquad \rho_{1,1}(\mathrm{q})=\begin{pmatrix} 0 & 0 \\ -\frac{1}{\sqrt 2} & 0 \end{pmatrix}\phantom{c}.
$$
Decomposition (\ref{dec}) is given by
$$
S=(\bR\mathrm{p}^{'}\otimes S_{0,n-1})+ (\bR 1^{'}\otimes S_{0,n-1}):=S_{-}+S_{+}\quad.
$$
Equation
$$\Cl_{1,n}\ni\frac{1}{2}\mathrm{p}e\overset{\approx}{\rightarrow}\frac{1}{2}(\mathrm{p}\otimes 1)\hat{\cdot}(1\otimes e)=\frac{1}{2}(\mathrm{p}\otimes e)\in\Cl_{1,1}\hat{\otimes}\Cl_{0,n-1}$$
implies that $\rho_{\spin}(\frac{1}{2}\mathrm{p}e)$ sends
$$
\mathrm{p}^{'}\otimes s_{0,n-1}	\longrightarrow 0\quad,\quad
1^{'}\otimes s_{0,n-1}\longrightarrow \frac{1}{2}(2\sqrt 2 \mathrm{p}^{'})\otimes e\cdot s_{0,n-1}
$$
for every $s_{0,n-1}\in S_{0,n-1}$.
\end{pf}
The following corollary is a direct consequence of Proposition \ref{pallino} and equation (\ref{ortogonale}).
\begin{corollary}
The image, under a spin representation $\rho:\gheis_{b}\rightarrow\ggl_{\bR}(S)$ of the Heisenberg algebra $\gheis_{b}=E^{*}+E+\bR\p$, of $e^{\flat}\in E^{*}$ is given by 
\be
\label{spinrapr}
\rho(e^{\flat})=\begin{pmatrix} 0 & \sqrt{2} Be \\ 0 & 0 \end{pmatrix}
\ee
where $Be\in E$ acts on $S_{0,n-1}$ by Clifford multiplication of a vector with a spinor.
\end{corollary}
For the sake of completness, recall that
$$
\Cl_{1,1}\hat{\otimes}\Cl_{0,n-1}=\bR\hat{\otimes}\Cl_{0,n-1}+\bR\p\hat{\otimes}\Cl_{0,n-1}+\bR\q\hat{\otimes}\Cl_{0,n-1}+\bR\p\q\hat{\otimes}\Cl_{0,n-1}
$$
and
$$
\rho_{\spin}(1\otimes c)=\begin{pmatrix} \alpha(c) & 0 \\ 0 & c \end{pmatrix}\phantom{cc},\phantom{c}\rho_{\spin}(\p\otimes c)=\begin{pmatrix} 0 & 2\sqrt{2}c \\ 0 & 0 \end{pmatrix}
$$
$$
\rho_{\spin}(\q\otimes c)=\begin{pmatrix} 0 & 0 \\ -\frac{1}{\sqrt{2}}\alpha(c) & 0 \end{pmatrix}\phantom{c},\phantom{c}
\rho_{\spin}(\p\q\otimes c)=\begin{pmatrix} -2\alpha(c) & 0 \\ 0 & 0 \end{pmatrix}
$$
for every $c\in\Cl_{0,n-1}$. 
\subsection{Aim of the paper}\hfill\newline\\
The goal of the paper is to give a description of superizations $\gg=\gg_{\0}+\gg_{\ou}=\gg_{b}+S$ of the CW algebra $\gg_{b}$ with translational supersymmetry. We find general results when the charge is zero and we prove that this is always the case if there are no semi-spinors. When semi-spinors exist, we give examples of odd-commutative superizations with non-zero charge. In particular we obtain a description of
\begin{itemize}
\item[(I)] Spin representations
$
\rho:\gheis_{b}\rightarrow\ggl_{\bR}(S)
$ of the Heisenberg algebra $\gheis_{b}$,
\item[(II)] Zero charge spin representations
$
\rho:\gg_{b}\rightarrow\ggl_{\bR}(S)
$
of the CW algebra $\gg_{b}$,
\item[(III)] Bilinear maps
$$
\Gamma:S\vee S\rightarrow\bR
$$
invariant by zero charge spin representations $\rho:\gg_{b}\rightarrow\ggl_{\bR}(S)$.
\end{itemize} 
The discrepancy between (III) and the hypothesis of translational supersymmetry is only apparent. Indeed denote by
$[\cdot,\cdot]:S\vee S\longrightarrow \gm$
the bilinear map which gives the Lie bracket between two odd elements $Q,\tilde{Q}\in S$ of a superization with translational supersymmetry. By abuse of notation, denote by
$$\p(Q,\tilde{Q})\quad,\quad\q(Q,\tilde{Q})\quad,\quad e_{i}(Q,\tilde{Q})$$ 
the corresponding components of $[Q,\tilde{Q}]\in\gm$, \textit{i.e.}
$$
[Q,\tilde{Q}]:=\p(Q,\tilde{Q})\mathrm{p}+\q(Q,\tilde{Q})\mathrm{q}+\sum_{i} e_{i}(Q,\tilde{Q})e_{i}\quad.
$$
\begin{lemma}
\label{bilinear}
Let $\gg=\gg_{b}+S$ be a superization with translational supersymmetry of a CW algebra $\gg_{b}=\gh+\gm$. Then
$
[Q,\tilde{Q}]=\p(Q,\tilde{Q})\mathrm{p}\in\bR\mathrm{p}
$.
\end{lemma}
\begin{pf}
The equation
$$
\gm\ni[[\mathrm{q},Q],\tilde{Q}]+[Q,[\mathrm{q},\tilde{Q}]]=[\mathrm{q},[Q,\tilde{Q}]]=\sum_{i}e_{i}(Q,\tilde{Q})e_{i}^{\flat}\in\gh
$$
implies that $e_{i}(Q,\tilde{Q})=0$ for $i=1, ...,n-1$ and hence that $[S,S]\subseteq\bR\mathrm{p}+\bR\mathrm{q}$. From this and the equation
$$
\bR\mathrm{p}+\bR\mathrm{q}\ni[[\mathrm{e_{i}^{\flat}},Q],\tilde{Q}]+[s,[\mathrm{e_{i}^{\flat}},\tilde{Q}]]=[\mathrm{e_{i}^{\flat}},[Q,\tilde{Q}]]=\q(Q,\tilde{Q})[\mathrm{e_{i}^{\flat}},\mathrm{q}]\in E
$$
it follows that $\q(Q,\tilde{Q})=0$.
\end{pf}
\section{Spin representation of the Heisenberg algebra}
\setcounter{equation}{0}
This section deals with (I); it describes the general structure of spin representations of the Heisenberg algebra, specializing it to the various dimensions. Results strongly depend on $\dim(E)=n-1\pc(\text{mod}\pc 8)$. This fact relies on the structure of the even Schur algebra which can be $\bR$, $\bC$ or $\bH$. Moreover we show that non-zero charge spin representations exist if and only if semi-spinors exist. By a representation
$
\rho:\gheis_{b}\rightarrow \ggl_{\bR}(S)
$
we will always mean a spin representation, \textit{i.e.} an extension of (\ref{spinrapr}) to the Heisenberg algebra.
\subsection{The image of the spin representation}\hspace{0 cm}\newline\\
This subsection starts the study of all (spin) representations of the Heisenberg algebra. The problem is reduced to a system of quadratic equations on the real division algebra $\cC_{0,n-1}^{\circ}$ (plus one extra-condition), whose solution will be the central point of next subsections.
\begin{lemma}
\label{p}
The images, under a representation $\rho:\gheis_{b}\rightarrow \ggl_{\bR}(S)$, of $\mathrm{p}\in\gheis$ and $e\in E$ are of the following form:
$$
\rho(\mathrm{p})=\begin{pmatrix} 0 & \mathrm{p}_{12} \\ 0 & 0 \end{pmatrix}
\qquad,\qquad
\rho(e)=\begin{pmatrix} \rho_{11}(e) & \rho_{12}(e) \\ 0 & \rho_{22}(e) \end{pmatrix}
$$
where $\mathrm{p}_{12}\in \End_{\bR}(S_{0,n-1})$ and $\rho_{11},\rho_{12},\rho_{22}\in\Hom_{\bR}(E,\End_{\bR}(S_{0,n-1}))$ satisfy the following conditions
\be
\label{gg}
b(f,e)
\begin{pmatrix}
\mathrm{p}_{12}Q_{+}\\
0 \\
\end{pmatrix}
=\sqrt 2
\begin{pmatrix}
-\rho_{11}(e)Bf\cdot Q_{+}+ Bf\cdot\rho_{22}(e)Q_{+}\\
0 \\
\end{pmatrix}
\ee
\be
\label{ggg}
\begin{pmatrix}
\mathrm{p}_{12}\rho_{22}(e)Q_{+}\\
0 \\
\end{pmatrix}
=
\begin{pmatrix}
\rho_{11}(e)\mathrm{p}_{12}Q_{+}\\
0 \\
\end{pmatrix}
\phantom{c}
\ee
\be
\label{abelian}
[\rho(e),\rho(f)]=0
\ee
for every $f\in E$.
\end{lemma}
\begin{pf}
Denote by
$$\rho(\mathrm{p})=\begin{pmatrix}
\mathrm{p}_{11} & \mathrm{p}_{12} \\ \mathrm{p}_{21} & \mathrm{p}_{22}
\end{pmatrix}
\qquad,\qquad
\rho(e)=\begin{pmatrix}
e_{11} & e_{12} \\ e_{21} & e_{22}
\end{pmatrix}
$$
the images under $\rho$ of $\mathrm{p}$ and $e\in E$; $f$ is an element of $E$.
Equation $[f^{\flat},\mathrm{p}]=0$ implies that
$$0=-[s,[f^{\flat},\mathrm{p}]]=-[\mathrm{p},\begin{pmatrix}
\sqrt 2 Bf\cdot Q_{+} \\ 0 
\end{pmatrix}]
+\begin{pmatrix}
0 & \sqrt 2 Bf \\ 0 & 0
\end{pmatrix}
[\mathrm{p},s]$$
which is equivalent to
$$
\begin{pmatrix}
\mathrm{p}_{11}Bf\cdot Q_{+} \\
\mathrm{p}_{21}Bf\cdot Q_{+} \\
\end{pmatrix}
=
\begin{pmatrix}
Bf\cdot(\mathrm{p}_{21}Q_{-}+\mathrm{p}_{22}Q_{+})\\
0 \\
\end{pmatrix}
$$
for any $f\in E$. It follows that $\mathrm{p}_{21}=0$ and
\be
\label{cav}
\mathrm{p}_{11}\circ f\cdot=f\cdot\circ\pc\mathrm{p}_{22}\qquad,\qquad\mathrm{p}_{22}\circ f\cdot=f\cdot\circ\pc\mathrm{p}_{11}
\ee
for any $f\in E$.
The equation $b(f,e)\mathrm{p}=[f^{\flat},e]$ implies that
$$b(f,e)
\begin{pmatrix}
\mathrm{p}_{11}Q_{-}+\mathrm{p}_{12}Q_{+}\\
\mathrm{p}_{22}Q_{+}\\
\end{pmatrix}
=\sqrt 2
\begin{pmatrix}
-e_{11}Bf\cdot Q_{+}+Bf\cdot e_{21}Q_{-}+ Bf\cdot e_{22}Q_{+}\\
-e_{21}Bf\cdot Q_{+}\\
\end{pmatrix}
$$
for any $e,f\in E$. For every $e,f\in E$ such that $b(f,e)=1$, it follows that $$\mathrm{p}_{11}=\sqrt{2}Bf\cdot\circ\pc e_{21}\qquad,\qquad\mathrm{p}_{22}=-\sqrt{2}e_{21}\circ Bf\cdot\phantom{cccccc}.$$ 
For every $Bf\in E$ such that $(Bf)^{2}=1$, this equation, together with (\ref{cav}), implies that
$
\mathrm{p}_{22}=\sqrt{2}e_{21}\circ Bf\cdot=0$ and $\mathrm{p}_{11}=0$. Therefore $e_{21}=0$ for all $e\in E$. The equations (\ref{ggg}), (\ref{abelian}) follow from $[\mathrm{p},e]=0$ and $[e,f]=0$.
\end{pf}
Equations (\ref{gg}) and (\ref{ggg}) are explicitly solved case by case and their investigation is equivalent to solving some quadratic equations on the real division algebra $\cC_{0,n-1}^{\circ}$. Then it is possible to reduce the extra condition (\ref{abelian}). Equations (\ref{gg}) and (\ref{ggg}) are re-written, in terms of the orthonormal basis of $E$, into a {\bf system of three equations}:
\\
\be
\label{laprima}
\pc\pc\pc 0=e_{i}\cdot\circ\pc\rho_{22}(e_{j})-\rho_{11}(e_{j})\circ e_{i}\cdot\qquad\forall\phantom{c}1\leq i\neq j\leq n-1
\ee
\be
\label{laseconda}
\frac{\mathrm{p}_{12}}{\sqrt2}=e_{j}\cdot\circ\pc\rho_{22}(e_{j})-\rho_{11}(e_{j})\circ e_{j}\cdot\qquad\forall\phantom{c} 1\leq j\leq n-1\quad\pc
\ee
\be
\label{laterza}
\mathrm{p}_{12}\circ\rho_{22}(e_{j})=\rho_{11}(e_{j})\circ\mathrm{p}_{12}\qquad\quad\pc\pc\forall\phantom{c} 1\leq j\leq n-1
\ee
The next subsections are dedicated to finding the solutions of this system, together with the extra-condition (\ref{abelian}). The following notion of suitability is useful to describe them.
\begin{definition}
\label{suitable}
\rm{Let $\rho_{11},\rho_{22}\in\Hom_{\bR}(E,\End_{\bR}(S_{0,n-1}))$ be fixed. A map
$$
\rho_{12}\in\Hom_{\bR}(E,\End_{\bR}(S_{0,n-1}))
$$
is called $(\rho_{11},\rho_{22})$-{\bf suitable} (or, by abuse of notation, suitable) if the bilinear map
$$
E\otimes E\rightarrow \End_{\bR}(S_{0,n-1})
$$
\be
\label{simmetria}
\phantom{ccccccccccccccccccc}e\otimes f\rightarrow\rho_{11}(e)\circ\rho_{12}(f)-\rho_{12}(f)\circ\rho_{22}(e)
\ee
is symmetric. A pair
\be
\label{vectorspace}
(\rho_{11},\rho_{12})\in\Hom_{\bR}(E,\cC_{0,n-1})\bigoplus \Hom_{\bR}(E,\End_{\bR}(S_{0,n-1}))
\ee
is called {\bf suitable} if $[\rho_{11}(E),\rho_{11}(E)]=0$ and it holds (\ref{simmetria}) with $\rho_{22}:=\overline{\rho_{11}}$.
}\end{definition}
The vector space $$\Hom_{\bR}(E,\End_{\bR}(S_{0,n-1}))$$ has a natural structure of $\so(E)$-module. Whenever the two linear maps $\rho_{11},\rho_{22}$ are $\so(E)$-invariant, the subspace of $\Hom_{\bR}(E,\End_{\bR}(S_{0,n-1}))$ which consists of $(\rho_{11},\rho_{22})$-suitable maps is $\so(E)$-stable. In this case, using representation theory of semisimple Lie algebras, it is possible to determine all $(\rho_{11},\rho_{22})$-suitable maps. An example of such description, for the non-zero charge case with $\dim E=n-1=8$, is given in subsection \ref{dieci}. \\
Suitable pairs are more difficult to study. For example $\rho_{11}\in\Hom_{\bR}(E,\cC_{0,n-1})$ is not $\so(E)$-invariant if it is not zero. Moreover, suitable pairs depend on the choice of an Abelian Lie subalgebra of the Schur algebra $\cC_{0,n-1}$. 
\subsection{Zero charge representation}\hspace{0 cm}\newline\\
We show that zero charge spin representations are in bijective correspondence with suitable pairs.
\begin{theorem}
\label{0c}
Every suitable pair (\ref{vectorspace}) defines a zero charge representation $\rho:\gheis_{b}\rightarrow\ggl_{\bR}(S)$ given by
\be
\label{common}
\rho(e^{\flat})=\begin{pmatrix} 0 & \sqrt{2} Be \\ 0 & 0 \end{pmatrix}\quad,\quad
\rho(\mathrm{p})=\begin{pmatrix} 0 & 0 \\ 0 & 0 \end{pmatrix}\quad,\quad
\rho(e)=\begin{pmatrix} \rho_{11}(e) & \rho_{12}(e) \\ 0 & \overline{\rho_{11}}(e) \end{pmatrix}
\ee
Moreover every zero charge representation of $\gheis$ is of this type.
\end{theorem}
\begin{proof}
It is easy to see that (\ref{common}) is a representation. Vice versa, since the charge is zero, the three equations (\ref{laprima}), (\ref{laseconda}), (\ref{laterza}) reduce to
$$
e_{i}\cdot\circ\pc\rho_{22}(e_{j})=\rho_{11}(e_{j})\circ e_{i}\cdot
$$
for {\bf any} $1\leq i,j\leq n-1$. Equation
$$
e_{i}e_{k}\cdot\circ\pc\rho_{11}(e_{j})= e_{i}\cdot\circ\pc\rho_{22}(e_{j})\circ e_{k}\cdot=\rho_{11}(e_{j})\circ e_{i}e_{k}\cdot
$$
implies that the linear map $\rho_{11}\in\Hom_{\bR}(E,\End_{\bR}(S_{0,n-1}))$ takes values in the Schur algebra $\cC_{0,n-1}$ and that $\rho_{22}=\overline{\rho_{11}}$. The condition of suitability is then equivalent to the extra condition $[\rho(E),\rho(E)]=0$, if we remark that
$$[\rho_{11}(E),\rho_{11}(E)]=0 \Longleftrightarrow [\overline{\rho_{11}}(E),\overline{\rho_{11}}(E)]=0\quad.$$
The theorem is hence proved.
\end{proof}
The next two subsections contain some special results about spin representations in Euclidean signature which we will use to solve the system of three equations (\ref{laprima}), (\ref{laseconda}), (\ref{laterza}).
\subsection{Spin representation in Euclidean signature}\hspace{0 cm}\newline\\
\label{euclidean}
Let us fix a $1\leq \i\leq n-1$. The isometric embedding
$$
\bR^{0,n-2}\hookrightarrow\bR^{0,n-1}
$$ 
$$
\quad(x_{1},...,x_{\i-1},x_{\i+1},..,x_{n-1})\rightarrow (x_{1},...,x_{\i-1},0,x_{\i+1},..,x_{n-1})
$$
given by $x_{\i}=0$ gives rise to a canonical embedding of Clifford algebras $$\Cl_{0,n-2}\hookrightarrow\Cl_{0,n-1}$$ which is explicitly given by
$$
\left\langle e_{1},..,e_{\i-1},\hat{e_{\i}},e_{\i+1},..,e_{n-1}\right\rangle=\Cl_{0,n-2}\subseteq\Cl_{0,n-1}=\left\langle e_{1},..,e_{\i-1},e_{\i},e_{\i+1},..,e_{n-1}\right\rangle\quad.
$$ 
This subsection studies the spin representation $$\rho_{0,n-1}:\Cl_{0,n-1}\longrightarrow \End_{\bR}(S_{0,n-1})$$ restricted to the even part $\Cl_{0,n-2}^{\circ}$ of $\Cl_{0,n-2}$, \textit{i.e.} it studies the injective map
\be
\label{iniettiva}
\rho_{0,n-1}|_{\Cl^{\circ}_{0,n-2}}:\Cl^{\circ}_{0,n-2}\longrightarrow \End_{\bR} (S_{0,n-1})
\ee
The goal of this subsection is to give a description of the vector space 
$$\End_{\bR}(S_{0,n-1})^{\Cl^{\circ}_{0,n-2}}$$
of linear endomorphisms of $S_{0,n-1}$ invariant by (\ref{iniettiva}). 
The representation (\ref{iniettiva}) is not irreducible but is the sum of two irreducible representations. Indeed consider the following involution of $S_{0,n-1}$
$$
e_{\i}\cdot:S_{0,n-1}\longrightarrow S_{0,n-1}
$$
and the projection operators
$
\pi^{\pm}:=\frac{1}{2}(\Id\pm e_{\i}\cdot)
$
which satisfy
$$
\pi^{+}+\pi^{-}=\Id\qquad,\qquad(\pi^{\pm})^{2}=\pi^{\pm}\qquad,\qquad\pi^{+}\pi^{-}=\pi^{-}\pi^{+}=0\phantom{c}.
$$
The $\pm 1$-eigenspaces $\pi^{\pm}S_{0,n-1}$ of $e_{\i}\cdot$ are obviously
$\Cl^{\circ}_{0,n-2}$-invariant and, for dimensional reasons, they are irreducible. The $\Cl^{\circ}_{0,n-2}$-module (\ref{iniettiva}) decomposes as direct sum
\be
\label{yy}
S_{0,n-1}=\pi^{+}S_{0,n-1}\oplus\pi^{-}S_{0,n-1}
\ee
of $\Cl^{\circ}_{0,n-2}$-irreducible modules. We prove the following
\begin{theorem}
\label{ttt}
Let $\varphi\in\End_{\bR}(S_{0,n-1})^{\Cl^{\circ}_{0,n-2}}$ be a ${\Cl^{\circ}_{0,n-2}}$-invariant endomorphism of $S_{0,n-1}$. Then, with respect to decomposition (\ref{yy}), 
$$
\varphi=\begin{pmatrix} a & b\phi \\ c\phi & d \end{pmatrix}
$$
where $a,b,c,d\in\cC^{\circ}_{0,n-1}$ are uniquely determined and
$
\phi:S_{0,n-1}^{\pm}\rightarrow S_{0,n-1}^{\mp}
$
is given by 
$$
\phi=
\left\{\begin{matrix}
\rho_{0,n-1}(\omega_{0,n-2})\qquad\rm{if}\qquad n-1\equiv 0\pc (mod\pc 2)\\
\phantom{cc}J\quad\quad\qquad\qquad\phantom{cccc}\rm{if}\qquad n-1\equiv 3,7\pc (mod\pc 8)\\
\phantom{cc}0\quad\quad\qquad\qquad\quad\phantom{cc}\rm{if}\qquad n-1\equiv 1,5 \pc (mod\pc 8) 
\end{matrix}\right.
$$
where $\omega_{0,n-2}\in\Cl_{0,n-2}$ is the volume form of $\Cl_{0,n-2}\subseteq\Cl_{0,n-1}$ and $J\in\cC_{0,n-1}$ is a complex structure if $n-1\equiv 3\pc (mod \pc 8)$ or a para-complex structure if $n-1\equiv 7\pc (mod\pc 8)$.
\end{theorem}
About the existence of such $J$ see subsection \ref{hhyy}. The proof of Theorem \ref{ttt} is an immediate consequence of the following four propositions.
\begin{proposition}
\label{e1}
If $\mathbf{n-1\equiv 0\pc (mod\pc 2)}$, the irreducible $\Cl^{\circ}_{0,n-2}$-modules $\pi^{\pm}S_{0,n-1}$ are equivalent. An equivalence is given 
by the following automorphism of $S_{0,n-1}$
$$
\rho_{0,n-1}(\omega_{0,n-2}):\pi^{\pm}S_{0,n-1}\longrightarrow\pi^{\mp}S_{0,n-1}
$$
where $\omega_{0,n-2}\in\Cl_{0,n-2}$ is the volume form of $\Cl_{0,n-2}\subseteq\Cl_{0,n-1}$.
\end{proposition}
\begin{proof}
It is a straightforward consequence of the first part of Lemma \ref{VOLUME}. 
\end{proof}
In the odd-dimensional case, the above method does not carry over as the volume form $\omega_{0,n-2}$ does not commute with the projection operators. 
\begin{proposition}
\label{e2}
If $\mathbf{n-1\equiv 3\pc (mod\pc 4)}$, the irreducible $\Cl^{\circ}_{0,n-2}$-modules $\pi^{\pm}S_{0,n-1}$ are equivalent. An equivalence is given 
by the following automorphism of $S_{0,n-1}$
\begin{itemize}
\item[] A complex structure $J\in\cC_{0,n-1}$ if $n-1\equiv 3 \pc (mod\pc 8)$, 
\item[] A paracomplex structure $J\in\cC_{0,n-1}$ if $n-1\equiv 7 \pc (mod\pc 8)$\quad.
\end{itemize} 
\end{proposition}
\begin{proof}
In both cases the endomorphism $J\in\cC_{0,n-1}$ is an element of the Schur algebra $\cC_{0,n-1}$ with invariant $\tau(J)=-1$ (see \cite{AC} and subsection \ref{hhyy}). In particular, it anticommutes with $e_{\i}\cdot$ and commutes with $\Cl_{0,n-2}^{\circ}$.
Note that the twisted center of the Clifford algebra $\Cl_{0,n-1}$ is trivial (see \cite{H}). It follows that the equivalence $J$ can {\bf not} be represented as an element of $\Cl_{0,n-1}$. 
\end{proof}
\begin{proposition}
\label{e3}
If $\mathbf{n-1\equiv 1\pc (mod\pc 4)}$, the irreducible $\Cl^{\circ}_{0,n-2}$-modules $\pi^{\pm}S_{0,n-1}$ are not equivalent.
\end{proposition}
\begin{pf}
It follows from injectivity of (\ref{iniettiva}) and from the fact that $\Cl^{\circ}_{0,n-2}$ is a semisimple matrix algebra.
\end{pf}
The previous propositions imply that 
$$
\varphi=\begin{pmatrix} a & b\phi \\ c\phi & d \end{pmatrix}
$$
where $a,b\in \End(\pi^{+}S_{0,n-1})^{\Cl_{0,n-2}^{\circ}}$ and $c,d\in\End(\pi^{-}S_{0,n-1})^{\Cl_{0,n-2}^{\circ}}$ are uniquely determined. To complete the proof of Theorem \ref{ttt} we need the following
\begin{proposition}
\label{domani}
The linear map
$$
\phantom{cccccc}\cC_{0,n-1}^{\circ}\rightarrow\End(\pi^{\pm}S_{0,n-1})^{\Cl^{\circ}_{0,n-2}}
$$
\be
\label{sur}
C\longrightarrow C|_{\pi^{\pm}S_{0,n-1}}
\ee
which sends an endomorphism $C\in\cC_{0,n-1}^{\circ}$ to its restriction to $\pi^{\pm}S_{0,n-1}$ is a linear isomorphism.
\end{proposition}
\begin{pf}
The map is well-defined because every $C\in\cC_{0,n-1}^{\circ}$ commutes with the projection operators and with $\Cl^{\circ}_{0,n-1}$. It is injective; indeed consider a non-zero element $C\in\cC^{\circ}_{0,n-1}\cong\bR,\bC$ or $\bH$ and suppose, for example, that $C|_{\pi^{+}S_{0,n-1}}$=0. Then $C$ is not invertible, which is absurd. Surjectivity follows by an investigation case by case which is worked out in the next subsection.
\end{pf} 
The reader who is not interested in the proof of the surjectivity of map (\ref{sur}) and in the description of the (even) Schur algebra depending on $\dim E=n-1\pc$ (mod $8$) can skip temporarily the next subsection.
\subsection{The real division algebra \texorpdfstring{$\cC_{0,n-1}^{\circ}$}{}}\hfill\newline\\
\label{hhyy}
This subsection describes the (even) Schur algebra in Euclidean signature. This provides the details missing in the proof of Proposition \ref{domani} and the background for all the main theorems on spin representations of the Heisenberg algebra proved in later sections. This information follows from \cite{AC}, where an admissible basis of the Schur algebra $\cC_{0,n-1}$ is constructed. In the following table, using the notation from \cite{AC}, we indicate when the basic element is proportional to the volume form $\omega\in\Cl_{0,n-1}$. In this case, the invariant $\tau$ is given in Lemma \ref{VOLUME}. In the opposite case, we indicate the value of the invariant $\tau\in\left\{+1,-1\right\}$. 
\\
\\
{\bf Table.} The value of $\tau$ for the admissible basis of the Schur algebra.
$$
\begin{array}{|c|c|c|c|c|c|c|c|c|c|c|}
\hline
n-1 (\text{mod } 8)&\cC_{0,n-1}&\cC_{0,n-1}^{\circ}&\Id&I&J&K=IJ&E&EI&EJ&EK \\
\hline
0& \bR\oplus\bR &\bR & +1& & & & \omega & & & \\
\hline
1& \bR &\bR & \omega& & & &  & & & \\
\hline
2& \bC &\bR & +1&  &\omega & &  & & & \\
\hline
3& \bH &\bC & +1& -\omega& -1& -1&  & & & \\
\hline
4& \bH\oplus\bH &\bH & +1& +1& -1& -1&  \omega& -1& +1& +1\\
\hline
5& \bH &\bH & \omega& +1& +1& +1&  & & & \\
\hline
6& \bC(2) &\bH & +1& +1& +1& +1& -1& \omega& -1& -1\\
\hline
7& \bR(2) &\bC & +1& -\omega& -1& -1&  & & & \\
\hline
\end{array}
$$
\\
The classification of Clifford algebras in \cite{LM} provides the type (real, complex, quaternionic) of the irreducible representation of $\Cl_{0,n-2}^{\circ}\cong\Cl_{n-3,0}$. This determines the real division algebra $\End(\pi^{\pm}S_{0,n-1})^{\Cl^{\circ}_{0,n-2}}$ and completes the proof of Proposition \ref{domani}. 
\cite{AC} proves that, in the case $n-1\equiv 7 \pc(\text{mod} 8)$, the automorphism $J$ is a para-complex structure satisfying $\left\{I,J\right\}=0$.

\subsection{Quadratic equations on the real division algebra \texorpdfstring{$\cC_{0,n-1}^{\circ}$}{}}\hfill\newline\\
With the help of Theorem \ref{ttt}, equations (\ref{laprima}), (\ref{laseconda}), (\ref{laterza}) are re-written into some quadratic equations on the real division algebra $\cC_{0,n-1}^{\circ}$. Whenever the even Schur algebra is quaternionic, we have a system of quadratic equations on $\bH$ whose solutions have been worked out in subsection \ref{QUAT}. As for the notation, for every complex number $c=a+ib$, the imaginary part $\Im c:=ib$ is defined like for quaternions. For any fixed $1\leq\i \leq n-1$, equation (\ref{laprima}) implies that \pc\pc\pc\pc\pc\pc\pc\pc\pc\pc\pc\pc\pc\pc\pc\pc
$$
\rho_{11}(e_{\i})\in\End_{\bR}(S_{0,n-1})^{\Cl^{\circ}_{0,n-2}}\ni\rho_{22}(e_{\i})\qquad.
$$
Due to Theorem \ref{ttt}, we can identify 
$$
\rho_{22}(e_{\i})\cong
\begin{pmatrix}a_{\i} & b_{\i}\phi \\ c_{\i}\phi & d_{\i}\end{pmatrix}\quad\quad a_{\i},b_{\i},c_{\i},d_{\i}\in\cC_{0,n-1}^{\circ}
$$
with respect to decomposition (\ref{yy}). We introduce the following {\bf new
variables}
\be
\label{campo}
h_{1}^{\i}:=(a_{\i}-d_{\i})\phantom{c},\phantom{c}h_{2}^{\i}:=(b_{\i}+c_{\i})\phantom{c},\phantom{c}h_{3}^{\i}:=(a_{\i}+d_{\i})\phantom{c},\phantom{c}h_{4}^{\i}:=(c_{\i}-b_{\i})
\ee
which are elements of the even Schur algebra $\cC_{0,n-1}^{\circ}$. The following two propositions give solutions of equations (\ref{laprima}) and (\ref{laseconda}).
\begin{proposition}
\label{schurm1}
If $n-1\equiv 0 \pc (mod\pc 2)$, then $\rho_{22}(e_{\i})$ and $\rho_{11}(e_{\i})$ satisfy (\ref{laprima}) if and only if there exist elements (\ref{campo}) of the even Schur algebra $\cC_{0,n-1}^{\circ}$ such that
$$
2\rho_{22}(e_{\i})=h_{3}^{\i}\Id+h_{1}^{\i}e_{\i}+(-1)^{\i+1}
h_{2}^{\i}e_{\i}\omega_{0,n-1}+(-1)^{\i}h_{4}^{\i}\omega_{0,n-1}
$$
and
$$
2\rho_{11}(e_{\i})=h_{3}^{\i}\Id-h_{1}^{\i}e_{\i}+(-1)^{\i+1}
h_{2}^{\i}e_{\i}\omega_{0,n-1}-(-1)^{\i}h_{4}^{\i}\omega_{0,n-1}
$$ 
Equation (\ref{laseconda}) is satisfied if and only if 
$$
\cC_{0,n-1}\ni\frac{\mathrm{p}_{12}}{\sqrt 2}= h_{1}^{\i}\Id+(-1)^{\i+1}h_{2}^{\i}\omega_{0,n-1}
$$
In particular 
$$
h_{1}:=h_{1}^{\i}\in\cC_{0,n-1}^{\circ}\quad,\quad h_{2}:=(-1)^{\i+1}h_{2}^{\i}\in\cC_{0,n-1}^{\circ}
$$
do not depend on $1\leq \i\leq n-1$.
\end{proposition}
\begin{pf}
The action of $2\rho_{22}(e_{\i})$ sends 
$$
s_{0,n-1}=\pi^{+}s_{0,n-1}+\pi^{-}s_{0,n-1}=\frac{1}{2}(\Id+e_{\i})s_{0,n-1}+\frac{1}{2}(\Id-e_{\i})s_{0,n-1}
$$
into 
$$
a_{\i}(\Id+e_{\i})s_{0,n-1}+b_{\i}\omega_{0,n-2}(\Id-e_{\i})s_{0,n-1}
+c_{\i}\omega_{0,n-2}(\Id+e_{\i})s_{0,n-1}+d_{\i}(\Id-e_{\i})s_{0,n-1}
$$
which equals to 
$$
a_{\i}s_{0,n-1}+a_{\i}e_{\i}s_{0,n-1}+b_{\i}\omega_{0,n-2}s_{0,n-1}-(-1)^{\i}b_{\i}\omega_{0,n-1}s_{0,n-1}+
$$
$$
d_{\i}s_{0,n-1}-d_{\i}e_{\i}s_{0,n-1}+c_{\i}\omega_{0,n-2}s_{0,n-1}+(-1)^{\i}c_{\i}\omega_{0,n-1}s_{0,n-1}
$$
Equations (\ref{laprima}) and (\ref{laseconda}) immediately imply the other results.
\end{pf}
The proof of the following proposition is similar and is omitted.
\begin{proposition}
\label{schu3}
If $n-1\equiv 3\pc (mod\pc 4)$, then $\rho_{22}(e_{\i})$ and $\rho_{11}(e_{\i})$ satisfy (\ref{laprima}) if and only if there exist elements (\ref{campo}) of the even Schur algebra $\cC_{0,n-1}^{\circ}$ such that
$$
2\rho_{22}(e_{\i})=h_{3}^{\i}\Id+h_{1}^{\i}e_{\i}+
h_{2}^{\i}J+h_{4}^{\i}J\circ e_{\i}
$$
and
$$
2\rho_{11}(e_{\i})=h_{3}^{\i}\Id-h_{1}^{\i}e_{\i}-
h_{2}^{\i}J+h_{4}^{\i}J\circ e_{\i}
$$
Equation (\ref{laseconda}) is satisfied if and only if 
$$
\cC_{0,n-1}\ni\frac{\mathrm{p}_{12}}{\sqrt 2}= h_{1}^{\i}\Id-h_{4}^{\i}J
$$
In particular 
$$h_{1}:=h_{1}^{\i}\in\cC_{0,n-1}^{\circ}\quad,\quad h_{4}:=h_{4}^{\i}\in\cC_{0,n-1}^{\circ}$$ 
do not depend on $1\leq \i\leq n-1$.
\end{proposition}
The following two theorems, together with the previous propositions, reduce the solution of the system of three equations to a system of quadratic equations on the even Schur algebra $\cC_{0,n-1}^{\circ}$ when $n-1\equiv 0$ (mod $2$) and $n-1\equiv 3$ (mod $4$). In the following proposition $m\in\bN$.
\begin{theorem}
Let $1\leq \i\leq n-1$ be fixed. Whenever $n-1=2m$, a pair $(\rho_{11}(e_{\i}),\rho_{22}(e_{\i}))$ which satisfies (\ref{laprima}) and (\ref{laseconda}) is a solution of (\ref{laterza}) if and only if 
\be
\label{sistemone}
\left\{\begin{matrix} [h_{1}^{\i},h_{3}^{\i}]+(-1)^{m+1}\left\{h_{2}^{\i},h_{4}^{\i}\right\}=0\\
(h_{1}^{\i})^{2}+(-1)^{m+1}(h_{2}^{\i})^{2}=0\phantom{cccccc}\\ [h_{1}^{\i},h_{2}^{\i}]=0\phantom{cccccccccccccccccccc}\\ 
\left\{h_{1}^{\i},h_{4}^{\i}\right\}+[h_{3}^{\i},h_{2}^{\i}]=0\phantom{ccccccccccc}
\end{matrix}\right.
\ee
where $h_{1}^{\i},..,h_{4}^{\i}$ are elements of the even Schur algebra $\cC_{0,n-1}^{\circ}$ given by (\ref{campo}).
\end{theorem}
\begin{pf}
Equation (\ref{laterza}) is equivalent to comparing
$$
[h_{1}^{\i}\Id+(-1)^{\i+1}h_{2}^{\i}\omega_{0,2m}]
\circ[h_{3}^{\i}\Id+h_{1}^{\i}e_{\i}
+h_{2}^{\i}\omega_{0,2m-1}+(-1)^{\i}h_{4}^{\i}\omega_{0,2m}]
$$
with
$$
[h_{3}^{\i}\Id-h_{1}^{\i}e_{\i}+h_{2}^{\i}\omega_{0,2m-1}-(-1)^{\i}h_{4}^{\i}\omega_{0,2m}]
\circ[h_{1}^{\i}\Id+(-1)^{\i+1}h_{2}^{\i}\omega_{0,2m}]\phantom{c}.
$$
The left hand side of the equation equals to
$$
h_{1}^{\i}h_{3}^{\i}\Id+(h_{1}^{\i})^{2}e_{\i}
+h_{1}^{\i}h_{2}^{\i}\omega_{0,2m-1}+(-1)^{\i}h_{1}^{\i}h_{4}^{\i}\omega_{0,2m}
$$
$$
+(-1)^{\i+1}h_{2}^{\i}h_{3}^{\i}\omega_{0,2m}-h_{2}^{\i}h_{1}^{\i}\omega_{0,2m-1}
+(-1)^{m+1}(h_{2}^{\i})^{2}e_{\i}+(-1)^{m+1}h_{2}^{\i}h_{4}^{\i}\Id
$$
while the right hand side is
$$
h_{3}^{\i}h_{1}^{\i}\Id-(h_{1}^{\i})^{2}e_{\i}
+h_{2}^{\i}h_{1}^{\i}\omega_{0,2m-1}-(-1)^{\i}h_{4}^{\i}h_{1}^{\i}\omega_{0,2m}
$$
$$
-(-1)^{\i}h_{3}^{\i}h_{2}^{\i}\omega_{0,2m}-h_{1}^{\i}h_{2}^{\i}\omega_{0,2m-1}
-(-1)^{m+1}(h_{2}^{\i})^{2}e_{\i}-(-1)^{m+1}h_{4}^{\i}h_{2}^{\i}\Id
$$ 
Examining (anti)-commutativity of the various terms with respect to $e_{\i}$ and $e_{i}$ ($i\neq \i$), the theorem follows.
\end{pf}
If $n-1\equiv 3$ (mod $4$), the even Schur algebra $\cC_{0,n-1}^{\circ}$ is isomorphic to $\bC=\bR+I\bR$. In this context
$$\overline{\phantom{c}}:\bC\cong\cC_{0,n-1}^{\circ}\rightarrow\cC_{0,n-1}^{\circ}\cong\bC$$ 
denotes the usual notion of conjugation (which can not be confused with (\ref{conf}) because (\ref{conf}) is the identity, when restricted to the even Schur algebra $\cC_{0,n-1}^{\circ}$).
\begin{theorem}
\label{notte}
Let $1\leq \i\leq n-1$ be fixed. Whenever $n-1\equiv 3\pc (mod\pc 4)$, a pair $(\rho_{11}(e_{\i}),\rho_{22}(e_{\i}))$ which satisfies (\ref{laprima}) and (\ref{laseconda}) is a solution of (\ref{laterza}) if and only if 
\be
\label{sistemonedisp}
\left\{\begin{matrix} h_{2}^{\i}\overline{h_{4}^{\i}}\in I\bR\\
(h_{1}^{\i})^{2}=J^{2}\cdot|h_{4}^{\i}|^{2}\phantom{cccccc}\\ 
h_{2}^{\i}\Re(h_{1}^{\i})+h_{4}^{\i}\Im(h_{3}^{\i})=0\\ 
h_{1}^{\i}h_{4}^{\i}-h_{4}^{\i}\overline{h_{1}^{\i}}=0
\end{matrix}\right.
\ee
where $h_{1}^{\i},..,h_{4}^{\i}$ are elements of the even Schur algebra $\cC_{0,n-1}^{\circ}\cong\bC$ given by (\ref{campo}).
\end{theorem}
\begin{pf}
Equation (\ref{laterza}) is equivalent to comparing
$$
[h_{1}^{\i}\Id-h_{4}^{\i}J]
\circ[h_{3}^{\i}\Id+h_{1}^{\i}e_{\i}+h_{2}^{\i}J+h_{4}^{\i}J\circ e_{\i}]
$$
with
$$
[h_{3}^{\i}\Id-h_{1}^{\i}e_{\i}-
h_{2}^{\i}J+h_{4}^{\i}J\circ e_{\i}]
\circ[h_{1}^{\i}\Id-h_{4}^{\i}J]\phantom{c}.
$$
The left hand side of the equation equals to
$$
h_{1}^{\i}h_{3}^{\i}\Id+(h_{1}^{\i})^{2}e_{\i}
+h_{1}^{\i}h_{2}^{\i}J+h_{1}^{\i}h_{4}^{\i}J\circ e_{\i}
-h_{4}^{\i}\overline{h_{3}^{\i}}J-h_{4}^{\i}\overline{h_{1}^{\i}}J\circ e_{\i}
-J^{2}\circ h_{4}^{\i}\overline{h_{2}^{\i}}\Id-J^{2}\circ h_{4}^{\i}\overline{h_{4}^{\i}}e_{\i}
$$
while the right hand side is
$$
h_{3}^{\i}h_{1}^{\i}\Id-(h_{1}^{\i})^{2}e_{\i}
-h_{2}^{\i}\overline{h_{1}^{\i}}J+(h_{4}^{\i})\overline{h_{1}^{\i}}J\circ e_{\i}
-h_{3}^{\i}h_{4}^{\i}J-h_{1}^{\i}h_{4}^{\i}J\circ e_{\i}
+J^{2}\circ h_{2}^{\i}\overline{h_{4}^{\i}}\Id+J^{2}\circ h_{4}^{\i}\overline{h_{4}^{\i}}e_{\i}.
$$
\end{pf}
\begin{remark}\rm{
In the case of zero charge, equations (\ref{sistemone}) and (\ref{sistemonedisp}) are trivially satisfied.
}\end{remark}
In the remaining case $n-1\equiv 1$ (mod $4$) it is proved that any representation has zero charge.
\begin{theorem}
\label{schurmdisp1}
If $n-1\equiv 1\pc (mod\pc 4)$, then every representation of the Heisenberg algebra has zero charge.
\end{theorem}
\begin{pf}
The endomorphisms $\rho_{22}(e_{\i})$ and $\rho_{11}(e_{\i})$ satisfy (\ref{laprima}) if and only if there exist elements (\ref{campo}) of the even Schur algebra $\cC_{0,n-1}^{\circ}$ satisfying
$$
2\rho_{22}(e_{\i})=h_{3}^{\i}\Id+h_{1}^{\i}e_{\i}\quad,\quad
2\rho_{11}(e_{\i})=h_{3}^{\i}\Id-h_{1}^{\i}e_{\i}\quad.
$$
Equation (\ref{laseconda}) is then satisfied if and only if 
$\frac{\mathrm{p}_{12}}{\sqrt 2}=h_{1}^{\i}\Id
$.
Equation (\ref{laterza}) implies that
$$
h_{1}^{\i}(h_{3}^{\i}\Id+h_{1}^{\i}e_{\i})=(h_{3}^{\i}\Id-h_{1}^{\i}e_{\i})h_{1}^{\i}\Rightarrow
\left\{\begin{matrix} [h_{1}^{\i},h_{3}^{\i}]=0\\
(h_{1}^{\i})^{2}=0\\ 
\end{matrix}\right.\Rightarrow h_{1}^{\i}=0
$$
from which the assertion of the theorem follows.
\end{pf}
The solutions of the systems (\ref{sistemone}) and (\ref{sistemonedisp}) are worked out, case by case, in the next two subsections. 
We first analyze the case of irreducible $\so(\gm)$-module $S$ and then the case when semi-spinors appear. 
\subsection{Case of irreducible spin module \texorpdfstring{$S$}{S}}\hspace{0 cm}\newline\\
When semi-spinors do not exist, a spin representation of the Heisenberg algebra has necessarily zero charge. Indeed, the following theorem holds. 
\begin{theorem}
\label{al}
If $\dim E=n-1\equiv 1,2,3,5\pc (mod\pc 8)$, then every representation of the Heisenberg algebra has zero charge.
\end{theorem} 
\begin{proof}
If $n-1\equiv 2$ (mod $8$), system (\ref{sistemone}) reduces to
$
h_{1}^{\i}=0=h_{2}^{\i}
$. If $n-1\equiv 3$ (mod $8$), system (\ref{sistemonedisp}) reduces to $h_{1}^{\i}=h_{4}^{\i}=0$. If $n-1\equiv 1,5$ (mod $8$), the theorem has been proved in Theorem \ref{schurmdisp1}. This completes the proof of the theorem. 
\end{proof}
These are precisely the dimensions for which semi-spinors do not exist. When semi-spinors exist, we show the presence of non-zero charge representations.
\subsection{Case of reducible spin module \texorpdfstring{$S$}{S}}\hspace{0 cm}\newline\\
Non-zero charge representations $\rho:\gheis_{b}\rightarrow\ggl_{\bR}(S)$ of the Heisenberg algebra are described in terms of elements $$h_{1},h_{2},h_{3},h_{4}$$ which are either elements of $\cC^{\circ}_{0,n-1}$ or linear maps from $E$ to $\cC_{0,n-1}^{\circ}$. In the following four theorems, the map $$\rho_{12}\in\Hom_{\bR}(E,\End_{\bR}(S_{0,n-1}))$$ is a suitable map (see Def. \ref{suitable}). 
\begin{theorem}
\label{0p}
If $n-1\equiv 0\pc (mod\pc 8)$, every non-zero charge representation of the Heisenberg algebra is given by
$$
\rho(\mathrm{p})=\sqrt 2\begin{pmatrix} 0 & h_{1}(\Id\pm E) \\ 0 & 0 \end{pmatrix}
$$
$$
\rho(e)=\frac{1}{2}\begin{pmatrix} h_{3}(e)\Id-h_{1}e\pm h_{1}eE & 2\rho_{12}(e) \\ 0 &h_{3}(e)\Id+h_{1}e\pm h_{1}eE  \end{pmatrix}
$$ 
where $h_{1}\in\bR$ satisfies $h_{1}\neq 0$ and $h_{3}\in\Hom_{\bR}(E,\bR)$.
\end{theorem}
\begin{pf}
System (\ref{sistemone}) reduces to  
$$\left\{\begin{matrix} h_{1}^{2}=h_{2}^{2}\\ \left\{h_{1},h_{4}^{\i}\right\}=\left\{h_{2},h_{4}^{\i}\right\}=0 \end{matrix}\right.$$
where $\cC_{0,n-1}^{\circ}\cong\bR$. The extra condition (\ref{abelian}) is equivalent to suitability of $\rho_{12}$. Proposition \ref{schurm1} then implies the result of the theorem.
\end{pf}

\begin{theorem}
\label{3p}
\label{mic}
If $n-1\equiv 6 \pc (mod\pc 8)$, every non-zero charge representation of the Heisenberg algebra is given by
$$
\rho(\mathrm{p})=\sqrt 2\begin{pmatrix} 0 & h_{1}\Id+h_{2}EI \\ 0 & 0 \end{pmatrix}
$$
$$
\rho(e)=\frac{1}{2}\begin{pmatrix} h_{3}(e)\Id-h_{1}e+& 2\rho_{12}(e) \\ +h_{2}eEI+h_{4}(e)EI  & \\ & h_{3}(e)\Id+h_{1}e+ & \\ 0 & +h_{2}eEI-h_{4}(e)EI \end{pmatrix}
$$
where either $h_{2}\in\bH$ satisfies $\Im h_{2}\neq 0$ and
$$\left\{\begin{matrix} h_{3}\in\Hom_{\bR}(E,\bH)\quad|\quad\dim\CoKer(\Im h_{3})\leq 1\quad{\it or}\quad \Im h_{3}(E)=(\Im h_{2})^{\bot}\\ h_{1}(h_{2},\pm)=\pm(|\Im h_{2}|- \Re h_{2}\frac{\Im h_{2}}{|\Im h_{2}|})\phantom{c};\phantom{c} h_{4}(h_{2},h_{3},\pm)(e)=\pm\frac{\Im h_{2}\times\Im h_{3}(e)}{|\Im h_{2}|} \end{matrix}\right.$$
or $h_{1}\in\bH$ satisfies $\Re h_{1}=0$, $\Im h_{1}\neq 0$ and
$$\left\{\begin{matrix} h_{3}\in\Hom_{\bR}(E,\bH)\quad|\quad\dim\CoKer(\Im h_{3})\leq 1\quad{\it or}\quad \Im h_{3}(E)=(\Im h_{1})^{\bot}\\ 
h_{2}(h_{1},\pm)=\pm|\Im h_{1}|\pc
\phantom{-\Re h_{1}\frac{\Im h_{1}}{|\Im h_{1}|})}
\phantom{c};\phantom{c}h_{4}(h_{1},h_{3},\pm)(e)=\mp\frac{\Im h_{1}\times\Im h_{3}(e)}{|\Im h_{1}|}\end{matrix}\right.$$
\end{theorem}
\begin{pf}
The system (\ref{sistemone}) is solved in Lemma \ref{quat2}. The condition of non-zero charge is given by $h_{1}^{\i}\neq 0$ and $h_{2}^{\i}\neq 0$ (see Proposition \ref{schurm1}). Suppose that 
$$
h_{1}=h_{1}^{\i}=\pm(|\Im h_{2}|-\Re h_{2}\frac{\Im h_{2}}{|\Im h_{2}|})
$$
where $\pm\in\left\{+1,-1\right\}$. Lemma \ref{quat2} implies that
$$
(-1)^{\i+1}h_{4}^{\i}=\pm(-1)^{\i+1}(-1)^{\i+1}\frac{\Im h_{2}\times \Im h_{3}^{\i}}{|\Im h_{2}|}=
\pm\frac{\Im h_{2}\times \Im h_{3}^{\i}}{|\Im h_{2}|}\quad.
$$
Suppose that
$$
(-1)^{\i+1}h_{2}=h_{2}^{\i}=\pm_{\i}|\Im h_{1}|
$$
where $\pm_{\i}\in\left\{+1,-1\right\}$ depends on $1\leq \i\leq n-1$, while $(-1)^{\i+1}\pm_{\i}:=\pm$ does not. Lemma \ref{quat2} implies that
$$
(-1)^{\i}h_{4}^{\i}=\pm_{\i}(-1)^{\i+1}\frac{\Im h_{1}\times\Im h_{3}^{\i}}{|\Im h_{1}|}=\pm\frac{\Im h_{1}\times\Im h_{3}^{\i}}{|\Im h_{1}|}\quad.
$$
The extra condition (\ref{abelian}) implies suitability of $\rho_{12}$ and
\be
\label{sono}
[h_{3}(e),h_{3}(f)]=[h_{4}(e),h_{4}(f)]\qquad,\qquad [h_{3}(e),h_{4}(f)]=[h_{3}(f),h_{4}(e)]
\ee
Equations (\ref{sono}) imply the stated conditions on the map $\Im h_{3}\in\Hom_{\bR}(E,\Im\bH)$ (details left to the reader). Proposition \ref{schurm1} then implies the result of the theorem.
\end{pf}


\begin{theorem}
\label{2p}
If $n-1\equiv 4 \pc (mod\pc 8)$, every non-zero charge representation of the Heisenberg algebra is given by
$$
\rho(\mathrm{p})=\sqrt 2\begin{pmatrix} 0 & h_{1}(\Id\pm E) \\ 0 & 0 \end{pmatrix}
$$
$$
\rho(e)=\frac{1}{2}\begin{pmatrix} h_{3}(e)\Id-h_{1}e+& 2\rho_{12}(e) \\ \pm h_{1}eE-h_{4}(e)E  & \\ & h_{3}(e)\Id+h_{1}e+ & \\ 0 & \pm h_{1}eE+h_{4}(e)E \end{pmatrix}
$$
where either $h_{1}\in\bH$ satisfies $\Re h_{1}\neq 0$ and
$$\left\{\begin{matrix} h_{3}\in\Hom_{\bR}(E,\bH)\quad|\quad \dim\CoKer(\Im h_{3})\leq 1\\ h_{4}(h_{1},\pm,h_{3})(e)=\mp\frac{\Im h_{1}\times\Im h_{3}(e)}{\Re h_{1}} \end{matrix}\right.$$
or $h_{1}\in\bH$ satisfies $\Re h_{1}=0$, $\Im h_{1}\neq 0$ and
$$\left\{\begin{matrix} h_{3}=(h_{3}^{\bR},q_{3}^{\bC})\in\Hom_{\bR}(E,\bC)\\ h_{4}\in\Hom_{\bR}(E,\bR)\quad|\quad\Ker(h_{4})=\Ker(h_{3}^{\bC})\quad{\it if}\quad h_{3}^{\bC}\neq 0\quad{\it and}\quad h_{4}\neq 0
\end{matrix}\right.$$
In the second case, the codomains of $h_{4}$ and $h_{3}$ are respectively identified with a 1-dimensional imaginary subspace of $(\Im h_{1})^{\bot}$ and with
$$\bC\cong\bR\oplus\bR\frac{\Im h_{1}}{|\Im h_{1}|}$$
\end{theorem}
\begin{pf}
The system (\ref{sistemone}) is solved in Lemma \ref{quat4}. We proceed similarly to the proof of Theorem \ref{mic}. Proposition \ref{schurm1} then implies the result of the theorem.
\end{pf}
\begin{theorem}
\label{3d}
If $n-1\equiv 7 \pc (mod\pc 8)$, every non-zero charge representation of the Heisenberg algebra is given by
$$
\rho(\mathrm{p})=\sqrt 2\begin{pmatrix} 0 & \pm|h_{4}|\Id-h_{4}J \\ 0 & 0 \end{pmatrix}
$$
$$
\rho(e)=\frac{1}{2}\begin{pmatrix} h_{3}(e)\Id\mp|h_{4}|e+h_{4}J\circ e & 2\rho_{12}(e) \\ 0 & h_{3}(e)\Id\pm|h_{4}|e+h_{4}J\circ e \end{pmatrix}
$$
where $h_{4}\in\bC$ satisfies $h_{4}\neq 0$ and
$
h_{3}\in\Hom_{\bR}(E,\bR)
$.
\end{theorem}
\begin{proof}
The second and fourth equations of system (\ref{sistemonedisp}) imply $$h_{1}^{\i}=\pm_{\i}|h_{4}^{\i}|$$
where $\pm_{\i}\in\left\{+1,-1\right\}$ depends on $1\leq \i\leq n-1$. The condition of non-zero charge is given by $h_{4}^{\i}\neq 0$ (see Proposition \ref{schu3}). The first and third equations of (\ref{sistemonedisp}) are then reduced to
$$h_{2}^{\i}=-\pm_{\i} \frac{h_{4}^{\i}}{|h_{4}^{\i}|}\Im h_{3}^{\i}\qquad.$$   
The extra condition $[\rho(E),\rho(E)]=0$ implies suitability of $\rho_{12}$ and $\Im h_{3}^{\i}=0$ (details left to the reader). Proposition \ref{schu3} then implies the result of the theorem.
\end{proof}
\subsection{Summary}\hfill\newline\\
We give a brief summary of the results obtained in this section. We have studied all possible spin representations of the Heisenberg algebra $\gheis_{b}=E^{*}+E+\bR\p$. Zero charge representations appear in all dimensions and are described in a unified way (see Theorem \ref{0c}). They correspond bijectively to suitable pairs (see Definition \ref{suitable}). Zero charge representations are the only spin representations of the Heisenberg algebra when semi-spinors do not exist (see Theorem \ref{al}). When semi-spinors do exist, all non-zero charge spin representations have been described in terms of suitable maps (see Theorems \ref{0p}, \ref{3p}, \ref{2p}, \ref{3d}).
\subsection{Solution of some quadratic equations on \texorpdfstring{$\bH$}{H}}\hspace{0 cm}\newline\\
\label{QUAT}
In this subsection we give solutions of some quadratic equations with quaternionic coefficients, which have been used in the case $n-1\equiv 6,4$ (mod $8$) in Theorem \ref{3p} and \ref{2p}.
\begin{lemma}\cite{Z}
\label{quat3}
For every $h, h_{1}, h_{2}\in\bH$ the following formulae are true
$$
[h_{1},h_{2}]=2(\Im h_{1}\times\Im h_{2})
$$
$$
\left\{h_{1},h_{2}\right\}=2(\Re h_{1}\Im h_{2}+\Re h_{2}\Im h_{1}+\Re h_{1}\Re h_{2}-\Im h_{1}\cdot\Im h_{2})
$$
$$
h^{2}=|\Re h|^{2}-|\Im h|^{2}+2\Re h\Im h
$$
and $h_{1}, h_{2}\in\bH$ are conjugate if and only if
$
\Re h_{1}=\Re h_{2}$, $|\Im h_{1}|=|\Im h_{2}|
$.
\end{lemma}
The trivial solutions of the following two lemmas correspond to zero charge spin representations of the Heisenberg algebra. This explains our terminology.
\begin{lemma}
\label{quat2}
There are three types of solutions of the following system
$$\left\{\begin{matrix} [h_{1},h_{3}]+\left\{h_{2},h_{4}\right\}=0\\
h_{1}^{2}+h_{2}^{2}=0\\ [h_{1},h_{2}]=0\\ 
\left\{h_{1},h_{4}\right\}+[h_{3},h_{2}]=0
\end{matrix}\right.$$
where $h_{i}\in\bH$ for $i=1,...,4$. They are given by 
\begin{itemize}
\item[1)] The zero charge solution $h_{1}=h_{2}=0$,\item[2)] The family $$h_{1}(h_{2},\pm)=\pm(|\Im h_{2}|- \Re(q_{2})\frac{\Im h_{2}}{|\Im h_{2}|})\quad,
\quad
h_{4}(h_{2},h_{3},\pm)=\pm\frac{\Im h_{2}\times\Im h_{3}}{|\Im h_{2}|}$$ when $\Im h_{2}\neq 0$,
\item[3)] The family $$h_{2}(h_{1},\pm)=\pm|\Im h_{1}|\phantom{cc}
\phantom{- \Re h_{1}\frac{\Im h_{1}}{|\Im h_{1}|})}
\quad,\quad\
h_{4}(h_{1},h_{3},\pm)=\mp\frac{\Im h_{1}\times\Im h_{3}}{|\Im h_{1}|}$$ when $\Im h_{1}\neq 0$ and $\Re h_{1}=0$. 
\end{itemize}
\end{lemma}
\begin{pf}
Suppose $\Re h_{1}=\Re h_{2}=0$. The second equation of the system implies that $\Im h_{1}=\Im h_{2}=0$, and then 
$h_{1}=h_{2}=0$. If $\Re h_{1}\neq 0$ (the other case is analogous), the second and third equations of the system imply that $\Im h_{1}$ and $\Im h_{2}$ are linearly dependent. More precisely,
\be
\label{ant}
\Re h_{1}\Im h_{1}+\Re h_{2}\Im h_{2}=0
\ee
Using (\ref{ant}), the second equation of the system implies that
$$
(1+\frac{(\Re h_{2})^{2}}{(\Re h_{1})^{2}})|\Im h_{2}|^{2}=(\Re h_{2})^{2}+(\Re h_{1})^{2}\qquad.
$$
This is equivalent to $|\Im h_{2}|^{2}=(\Re h_{1})^{2}\neq 0$. Summarizing we get that
$$h_{1}=\pm(|\Im h_{2}|-\Re h_{2}\frac{\Im h_{2}}{|\Im h_{2}|})\phantom{c}.$$
Similarly, we get the case 
$$h_{2}=\pm(|\Im h_{1}|-\Re h_{1}\frac{\Im h_{1}}{|\Im h_{1}|})$$
when $(\Re h_{2})^{2}=|\Im h_{1}|^{2}\neq 0$. In the case $\Im h_{2}\neq 0$ (\textit{i.e.} $\Re h_{1}\neq 0$), the first and fourth equations of the system imply
\be
\label{stanco1}
\mp\Re h_{2}\frac{\Im h_{2}}{|\Im h_{2}|}\times\Im h_{3}+\Re h_{2}\Im h_{4}+\Re h_{4}\Im h_{2}=0
\ee
\be
\label{stanco2}
\Re h_{2}\Re h_{4}-\Im h_{2}\cdot\Im h_{4}=0
\ee
\be
\label{stanco3}
\Im h_{3}\times\Im h_{2}\pm|\Im h_{2}|\Im h_{4}\mp\Re h_{4}\Re h_{2}\frac{\Im h_{2}}{|\Im h_{2}|}=0
\ee
\be
\label{stanco4}
|\Im h_{2}|\Re h_{4}+\Re h_{2}\frac{\Im h_{2}}{|\Im h_{2}|}\cdot\Im h_{4}=0
\ee
It is easy to see that (\ref{stanco2}) and (\ref{stanco4}) are equivalent to
$
\Re h_{4}=\Im h_{2}\cdot\Im h_{4}=0
$.
Equations (\ref{stanco1}), (\ref{stanco3}) are reduced then to
$$
\mp\Re h_{2}\frac{\Im h_{2}}{|\Im h_{2}|}\times\Im h_{3}+\Re h_{2}\Im h_{4}=0
$$
$$
\Im h_{3}\times\Im h_{2}\pm|\Im h_{2}|\Im h_{4}=0
$$
which are equivalent to
$\Im h_{4}=\pm\frac{\Im h_{2}\times\Im h_{3}}{|\Im h_{2}|}$. The lemma is thus proved.
\end{pf}
\begin{lemma}
\label{quat4}
There are three types of solutions of the following system 
$$\left\{\begin{matrix} [h_{1},h_{3}]-\left\{h_{2},h_{4}\right\}=0\\
h_{1}^{2}=h_{2}^{2}\\ [h_{1},h_{2}]=0\\ 
\left\{h_{1},h_{4}\right\}+[h_{3},h_{2}]=0
\end{matrix}\right.$$
where $h_{i}\in\bH$ for $i=1,...,4$. They are given by
\begin{itemize}
\item[1)] The zero charge solution $h_{1}=h_{2}=0$,
\item[2)] The family $$h_{2}(h_{1},\pm)=\pm h_{1}\quad,\quad h_{4}(h_{1},h_{3},\pm)=\pm\frac{\Im h_{1}\times\Im h_{3}}{\Re h_{1}}
$$
when $\Re h_{1}\neq 0$,
\item[3)] The family $$h_{2}(h_{1},\pm)=\pm h_{1}\phantom{c},\phantom{c}\Re h_{4}=0\phantom{c},\phantom{c}\Im h_{4}\bot\Im h_{1}\phantom{c},\phantom{c}\Im h_{3}\in\bR\cdot\Im h_{1}
$$
when $\Im h_{1}\neq 0$ and $\Re h_{1}=0$.
\end{itemize}
\end{lemma}
\begin{pf}
Suppose $\Re h_{1}\neq0$ (the other case $\Re h_{2}\neq 0$ is analogous). The second and third equations of the system imply that $\Im h_{1}$ and $\Im h_{2}$ are linearly dependent. More precisely,
\be
\label{ahah}
\Re h_{1}\Im h_{1}-\Re h_{2}\Im h_{2}=0
\ee
Using (\ref{ahah}), 
the second equation of the system implies that
$$
(1-\frac{(\Re h_{2})^{2}}{(\Re h_{1})^{2}})|\Im h_{2}|^{2}=(\Re h_{2})^{2}-(\Re h_{1})^{2}\qquad.
$$
Since the two sides have different signs, it follows that $(\Re h_{2})^{2}=(\Re h_{1})^{2}$. Equation (\ref{ahah}) finally implies that
\be
\label{finally} 
h_{1}=\pm h_{2}
\ee
Suppose $\Re h_{1}=\Re h_{2}=0$. The second and third equations of the system imply that
$\Im h_{1}$ and $\Im h_{2}$ are linearly dependent vectors that lie on the same $2$-dimensional sphere. Equation (\ref{finally}) then holds. The first and fourth equations of the system are given by 
\be
\label{nonso}
[h_{1},h_{3}]=+\left\{h_{1},h_{4}\right\}\qquad\text{if}\quad h_{1}=+h_{2}
\ee
\be
\label{nonso2}
[h_{1},h_{3}]=-\left\{h_{1},h_{4}\right\}\qquad\text{if}\quad h_{1}=-h_{2}
\ee
Here only the case (\ref{nonso}) is studied (the other case is analogous). Since $h_{1}\neq 0$, equation (\ref{nonso}) is re-written as
\be
\label{equivalent3}
h_{1}^{-1}q_{1}h_{1}=q_{2}
\ee
where $q_{1}:=h_{3}+h_{4}$ and $q_{2}:=h_{3}-h_{4}$ are two conjugate quaternions. Lemma \ref{quat3} implies that
$$
\Re h_{4}=0\qquad,\qquad\Im h_{4}\bot\Im h_{3}\phantom{c}
$$
and then equation (\ref{equivalent3}) reduces to
$
h_{1}(\Im h_{3}-\Im h_{4})=(\Im h_{3}+\Im h_{4})h_{1}
$, \textit{i.e.}
$$
\Im h_{1}\bot\Im h_{4}\qquad,\qquad\Re h_{1}\Im h_{4}=\Im h_{1}\times\Im h_{3}\phantom{c}.
$$
The solution of this equation depends on $\Re(h_{1})$ as follows
$$ \left\{\begin{matrix}\quad\pc\pc\pc\pc \Re h_{1}\neq 0 \qquad,\quad\quad\quad \Im h_{4}=+\frac{\Im h_{1}\times\Im h_{3}}{\Re h_{1}}\qquad\qquad\qquad\quad\bf{or}\\
\Re h_{1}=0\qquad,\qquad \Im h_{3}\in\bR\cdot\Im h_{1}\quad,\quad\Im h_{4}\bot\Im h_{1}\\ 
\end{matrix}\right.$$
The lemma is thus proved.
\end{pf}

\section{Superization of the Heisenberg algebra}
\setcounter{equation}{0}
This section gives some results about superizations of the Heisenberg algebra, both in the zero charge and in the non-zero charge case. 
In the first case we focus mainly on superizations such that the action of $E$ on $S$ is upper triangular and $[S,S]\subseteq\bR\p$. This is due to the fact that these conditions are satisfied for all zero charge superizations of the CW algebra with translational supersymmetry (see Lemma \ref{bilinear} and Proposition \ref{ne}). In the second case, we restrict ourselves to $\dim E=8$.\\
The following lemma describes the form of $E^{*}$-invariant symmetric bilinear forms on the spin module $S=S_{-}+S_{+}$. According to (\ref{dec}) any symmetric bilinear form $\Gamma\in\Bil_{\bR}(S)$ can be written as $\Gamma=\Gamma_{--}+\Gamma_{-+}+\Gamma_{++}$ where
\be
\label{r--}
\Gamma_{--}:S_{-}\vee S_{-}\rightarrow\bR
\ee
\be
\label{r-+}
\Gamma_{-+}:S_{-}\otimes S_{+}\rightarrow\bR
\ee
\be
\label{r++}
\Gamma_{++}:S_{+}\vee S_{+}\rightarrow\bR
\ee
are defined by restrictions. Recall that $S_{\mp}$ is linear isomorphic to $S_{0,n-1}$.
\begin{lemma}
\label{MM}
Let $\Gamma\in\Bil_{\bR}(S)$ be a symmetric bilinear form on the spin module $S$ satisfying
$$\Gamma(\left(
\begin{array}{c}
e\cdot Q_{+} \\
0 \\
\end{array}
\right),
\left(
\begin{array}{c}
\tilde{Q}_{-} \\
\tilde{Q}_{+} \\
\end{array}
\right))+
\Gamma(\left(
\begin{array}{c}
Q_{-} \\
Q_{+} \\
\end{array}
\right)
\left(
\begin{array}{c}
e\cdot \tilde{Q}_{+} \\
0 \\
\end{array}
\right))=0\qquad\qquad\forall e\in E 
\qquad.$$
Then $\Gamma_{--}=0$, $\Gamma_{-+}\in\Bil_{\bR}(S_{0,n-1})^{\tau\sigma=-1}$ and $\Gamma_{++}$ is an arbitrary symmetric bilinear form on $S_{0,n-1}$.
\end{lemma}
\begin{pf}
Choosing $Q_{-}=0=\tilde{Q}_{+}$, it follows that
$$\Gamma(\left(
\begin{array}{c}
e\cdot Q_{+} \\
0 \\
\end{array}
\right),
\left(
\begin{array}{c}
\tilde{Q}_{-} \\
0 \\
\end{array}
\right))=0
$$
\textit{i.e.} $\Gamma_{--}=0$. Choosing $Q_{-}=0=\tilde{Q}_{-}$, it follows that
$$0=\Gamma(\left(
\begin{array}{c}
e\cdot Q_{+} \\
0 \\
\end{array}
\right),
\left(
\begin{array}{c}
0 \\
\tilde{Q}_{+} \\
\end{array}
\right))+
\Gamma(\left(
\begin{array}{c}
0 \\
Q_{+} \\
\end{array}
\right)
\left(
\begin{array}{c}
e\cdot \tilde{Q}_{+} \\
0 \\
\end{array}
\right))=
$$
$$\phantom{cc}\Gamma(\left(
\begin{array}{c}
e\cdot Q_{+} \\
0 \\
\end{array}
\right),
\left(
\begin{array}{c}
0 \\
\tilde{Q}_{+} \\
\end{array}
\right))+
\Gamma(\left(
\begin{array}{c}
e\cdot \tilde{Q}_{+} \\
0 \\
\end{array}
\right),
\left(
\begin{array}{c}
0 \\
Q_{+} \\
\end{array}
\right)
)
$$
\textit{i.e.}
$\Gamma_{-+}(e\cdot Q_{+},\tilde{Q}_{+})=-
\Gamma_{-+}(e\cdot \tilde{Q}_{+},Q_{+})
$.
It follows that
$$
\Gamma_{-+}(e_{i}e_{j}\cdot Q_{+},\tilde{Q}_{+})=-\Gamma_{-+}(e_{i}\cdot\tilde{Q}_{+},e_{j}\cdot Q_{+})
=\Gamma_{-+}(e_{j}e_{i}e_{j}\cdot\tilde{Q}_{+},e_{j}\cdot Q_{+})=
$$
$$
-\Gamma_{-+}(e_{j}e_{j}\cdot Q_{+},e_{i}e_{j}\cdot\tilde{Q}_{+})
=-\Gamma_{-+}(Q_{+},e_{i}e_{j}\cdot\tilde{Q}_{+})
$$
for every $1\leq i\lneq j\leq n-1$, \textit{i.e.} $\Gamma_{-+}$ is an $\so(0,n-1)$-invariant form. We calculate the invariants. Decompose $\Gamma_{-+}$ into a direct sum 
$$
\Gamma_{-+}=\sum_{k}\Gamma_{-+}^{k}
$$ 
of admissible forms $\Gamma_{-+}^{k}$ with invariants $(\tau(k),\sigma(k))$ and get that
$$
\sum_{k}{\tau(k)\sigma(k)}\Gamma_{-+}^{k}(e\cdot \tilde{Q}_{+},Q_{+})=\Gamma_{-+}(e\cdot Q_{+},\tilde{Q}_{+})=
$$
$$
-\Gamma(e\cdot \tilde{Q}_{+},Q_{+})=-\sum_{k}\Gamma_{-+}^{k}(e\cdot \tilde{Q}_{+},Q_{+})
$$
for every $Q_{+},\tilde{Q}_{+}\in S_{0,n-1}$. The lemma is thus proved.
\end{pf}
\subsection{Zero charge superization}\hfill\newline\\
We consider zero charge superizations, of the Heisenberg algebra $\gheis_{b}=E^{*}+E+\bR\p$,
$$\gg=\gg_{\0}+\gg_{\ou}=\gheis_{b}+S$$
such that the action of $E$ on $S$ is upper triangular and $[S,S]\subseteq\bR\p$, \textit{i.e.}
\be
\label{fftt}
\rho(e)=\begin{pmatrix} 0 & \rho_{12}(e) \\ 0 & 0 \end{pmatrix}
\ee
\be
\label{ddee}
[Q,\tilde{Q}]=\p(Q,\tilde{Q})\mathrm{p}\pc\pc\pc\pc\pc
\ee
where $e\in E$ and $Q,\tilde{Q}\in S$. Here $\p\in\Bil_{\bR}(S)$ is a symmetric $\gheis_{b}$-invariant bilinear form on S. The following proposition holds.
\begin{proposition}
\label{zzz}
If $\dim E=n-1$, every zero charge superization of the Heisenberg algebra $\gheis_{b}=E^{*}+E+\bR\p$ satisfying (\ref{fftt}) and (\ref{ddee}) is uniquely determined by
$$
\rho_{12}\in\Hom_{\bR}(E,\End_{\bR}(S_{0,n-1}))\phantom{c},\phantom{c} \p_{-+}\in\Bil_{\bR}(S_{0,n-1})^{\tau\sigma=-1}\phantom{c},\phantom{c} p_{++}\in\Bil_{\bR}(S_{0,n-1})
$$ 
where $p_{++}$ is a symmetric bilinear form on $S_{0,n-1}$ such that 
$$
\p_{-+}(\rho_{12}(e)\cdot,\cdot)\in\Bil_{\bR}(S_{0,n-1})
$$
is skew-symmetric for every $e\in E$.       
\end{proposition}
\begin{proof}
This is a direct consequence of Lemma \ref{MM} and the condition of $E$-invariance for the bracket (\ref{ddee}).
\end{proof}
The interested reader can construct plenty of examples of zero charge superizations of the Heisenberg algebra. To make the paper easier to read, we do not pursue this point but, for the sake of completness, we give one example. Consider the case $n-1\equiv 2$ (mod $8$). 
\begin{proposition}
\label{zzzz}
If $n-1\equiv 2 \pc (mod\pc 8)$, every zero charge superization of Heisenberg algebras $\gheis_{b}=E^{*}+E+\bR\p$ with $[S,S]=\bR\p$ is such that
$$
\p_{--}=\p_{-+}=0\qquad,\qquad \p_{++}\neq 0\qquad,\qquad\rho_{12}\phantom{cc}{\it suitable}
$$
and
$$
\rho_{11}=0\qquad{\it or}\qquad \Hom(E,J\bR)\ni\rho_{11}\neq 0\quad,\quad \p_{++}\in\Bil_{\bR}(S_{0,n-1})^{J}
$$
where $\Bil_{\bR}(S_{0,n-1})^{J}$ is the space of $J$-hermitian bilinear form on $S_{0,n-1}$.
\end{proposition}
\begin{pf}
Lemma \ref{MM} implies that $\p_{--}=\p_{-+}=0$. Indeed all the admissible forms on $S_{0,n-1}$ have invariants $(\tau,\sigma)$ such that $\tau\sigma=1$ (see \cite{AC}). Write the condition of $E$-invariance of the bracket (\ref{ddee}) as:
$$\p_{++}(\overline{\rho_{11}}(e)Q_{+},\tilde{Q}_{+})+
\p_{++}(Q_{+},\overline{\rho_{11}}(e)\tilde{Q}_{+})=0\phantom{c}.
$$
Denote by
$
\rho_{11}:=\rho_{11}^{\bR}+\rho_{11}^{\bC}
$
the decompostion of $\rho_{11}\in\Hom_{\bR}(E,\cC_{0,n-1})$ in real and imaginary parts of $\cC_{0,n-1}\cong\bC=\bR+J\bR$. It is easy to prove that
\begin{itemize}
\item[$\cdot$] If $\rho_{11}\equiv 0$ then there are no costraints on $\p_{++}$,
\item[$\cdot$] If $\rho_{11}^{\bR}\neq 0= \rho_{11}^{\bC}$ then $\p_{++}=0$,
\item[$\cdot$] If $\rho_{11}^{\bR}=0\neq \rho_{11}^{\bC}$ then $\p_{++}$ must be $J-$hermitian,
\item[$\cdot$] If $\rho_{11}^{\bR}\neq 0$, $\rho_{11}^{\bC}\neq 0$ then $\p_{++}=0$.
\end{itemize}
The proposition is thus proved.
\end{pf}
\subsection{Non-zero charge superization}\hfill\newline\\
\label{dieci}
Theorem \ref{al} implies that there exist superizations of the Heisenberg algebra $\gheis_{b}$, with non-zero charge representation, only when the $\so(0,n-1)$-module $S_{0,n-1}$ is not irreducible. In "low dimension", this happens when $\dim E=n-1=4,6,7,8$. We study the last case \textit{i.e.} $E\cong \bR^{0,8}$. By Theorem \ref{0p}, we need to determine which linear maps
\be
\label{sui}
\rho_{12}:\bR^{0,8}\rightarrow\End_{\bR}(S_{0,8})
\ee
are $(\rho_{11},\rho_{22})$-suitable. In this case, the condition does not depend on the choice of the pair $(\rho_{11},\rho_{22})$, \textit{i.e.} on $h_{3}\in\Hom_{\bR}(E,\bR)$ and $h_{1}\in\bR-\left\{0\right\}$. Therefore we can speak of suitable maps. 
Let $E\in\cC_{0,8}$ be the para-complex structure described in Theorem \ref{0p} (unluckily this has the same notation of the Euclidean space $E$ but there should be no danger of confusion) and let
\be
\label{subito}
S_{0,8}=S_{0,8}^{+}\oplus S_{0,8}^{-}
\ee
be the decomposition of the spin module $S_{0,8}$ into the $\pm 1$-eigenspaces of $E$, \textit{i.e.} the decomposition (\ref{semi}) of $S_{0,8}$ into (inequivalent) semi-spin modules. Since $\tau(E)=-1$, Clifford multiplication by a vector $e\in E$ interchanges the two eigenspaces
and recall that $S_{0,8}^{\pm}$ is a self-dual module \textit{i.e.} $S_{0,8}^{\pm}\cong (S_{0,8}^{\pm})^{*}$
as $\so(0,8)$-modules.
The decomposition (\ref{dec}) of the spin module $S=S_{1,9}$ in Lorentzian signature into spin modules in Euclidean signature
$$
S=S_{-}+S_{+}= S_{0,8}+S_{0,8}\ni\begin{pmatrix} Q_{-}\\ Q_{+} \end{pmatrix}
$$
is then further refined to
$$
S=(S_{-}^{+}+S_{-}^{-})+(S_{+}^{+}+S_{+}^{-})\ni\begin{pmatrix} Q_{-}^{+} \\ Q_{-}^{+} \\ Q_{+}^{+} \\ Q_{+}^{-} \end{pmatrix}
$$
where the upper sign refers to (\ref{subito}). Recall that the supervector space
$$
\gheis_{b} + S=(E^{*}+E+\bR\mathrm{p})+(S_{0,8}+S_{0,8})
$$
becomes a superization of the Heisenberg algebra, with non-zero charge representation, when the action of the even part $\gheis_{b}$ on the odd part $S$ is given by
$$
\rho(e^{\flat})=\sqrt{2} \begin{pmatrix} 0 & Be \\ 0 & 0 \end{pmatrix}
\phantom{c},\phantom{c}
\rho(\mathrm{p})=2\sqrt{2}h_{1}\begin{pmatrix} 0 & \pi^{+} \\ 0 & 0 \end{pmatrix}
$$
$$
\rho(e)=\frac{1}{2}\begin{pmatrix} h_{3}(e)\Id-2h_{1}e\circ\pi^{-} & 2\rho_{12}(e) \\ 0 & h_{3}(e)\Id+2h_{1}e\circ\pi^{+} \end{pmatrix}
$$
where $\pi^{\pm}:=\frac{1}{2}(\Id\pm E)$ denotes projection on the semi-spin module $S_{0,8}^{\pm}$. The case with opposite sign is completely analogous. In order to solve the condition of suitability, we need to fix the notations. The irreducible $\so(E)\cong\so(0,8)$ modules are denoted as in \cite{OV}: for example
$$
R(\pi_{1}):=\Lambda^{1}E\quad,\quad R(\pi_{2}):=\Lambda^{2}E\quad,\quad R(\pi_{3}):=S_{0,8}^{+}\quad,\quad R(\pi_{4}):=S_{0,8}^{-} 
$$
are the irreducible representations associated with the nodes of the Dynkin diagram
$$
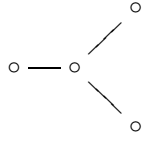
\xymatrix@C=1pc@R=1pc{
& & \circ\\
\circ\ar@{-}[r]& \circ\ar@{-}[ur]\ar@{-}[dr] & \\
& & \circ
}
$$
and more generally $R(\Lambda)$ is the irreducible module of $\so(E)$ with highest weight $\Lambda$ (with the exception of $R(\hat{\pi}_{4}):=\Lambda^{4}E$ which is reducible). In the notation of \cite{OV}
$$
R(\hat{\pi}_{p}):=\Lambda^{p}E
$$
for every $1\leq p\leq 8$ and in particular
$$
\Lambda^{4}E=R(\hat{\pi}_{4})=R(2\pi_{3})\oplus R(2\pi_{4})=\Lambda^{4}_{+}E\oplus\Lambda^{4}_{-}E
$$
is the decomposition of $\Lambda^{4}E$ into self-dual and anti self-dual forms. The main theorem of this subsection is the following.
\begin{theorem}
\label{aaa}
If $\dim E=8$, non-zero charge superizations of the Heisenberg algebra $\gheis_{b}=E^{*}+E+\bR\p$ with fixed $h_{1}\in\bR-\left\{0\right\}$ are in bijective correspondence with elements of the $\so(0,8)$-module
$$
R(\pi_{1})\oplus R(\pi_{1})\oplus\bR\oplus R(2\pi_{1})\oplus R(\pi_{1}+\pi_{2})\oplus R(\hat{\pi}_{3})\oplus R(\pi_{1})\oplus R(\pi_{2})\oplus R(2\pi_{4})\quad.
$$
Every non-zero charge superization is odd-commutative.
\end{theorem}
The proof of the theorem is accomplished through a careful analysis of all possible suitable maps (\ref{sui}) and of all possible brackets between odd elements. Theorem {\ref{aaa}} follows from the two propositions below and the remark that $h_{3}\in\Hom_{\bR}(E,\bR)\cong R(\pi_{1})$.
\begin{proposition}
\label{S}
The $\so(0,8)$-module which consists of suitable maps (\ref{sui}) is isomorphic to
$$
R(\pi_{1})\oplus\bR\oplus R(2\pi_{1})\oplus R(\pi_{1}+\pi_{2})\oplus R(\hat{\pi}_{3})\oplus R(\pi_{1})\oplus R(\pi_{2})\oplus R(2\pi_{4})\quad.
$$
\end{proposition}
\begin{proof}
Here we denote by $S=S_{0,8}$ the spin module $S_{0,8}$ in Euclidean signature.
The condition of suitability for $\rho_{12}\in E^{*}\otimes S^{*}\otimes S$ can be translated to be in the kernel of the following three $\so(0,8)$-invariant linear maps:
$$
E^{*}\otimes S^{*}\otimes S\longrightarrow \Lambda^{2}E^{*}\otimes (S^{+})^{*}\otimes S^{+}\phantom{ccccccccccccccccccccccccccccccccccc}
$$
\be
\label{11}
\rho_{12}\longrightarrow \pi^{+}\circ(e\wedge f\rightarrow \rho_{12}(e)\circ f\cdot-\rho_{12}(f)\circ e\cdot- e\cdot\circ \rho_{12}(f)+f\cdot\circ\rho_{12}(e))|_{S_{+}}
\ee
$$
E^{*}\otimes S^{*}\otimes S\longrightarrow \Lambda^{2}E^{*}\otimes (S^{+})^{*}\otimes S^{-}\phantom{ccccccccccccccccccccccccccccccccccc}
$$
\be
\label{22}
\rho_{12}\longrightarrow \pi^{-}\circ(e\wedge f\rightarrow \rho_{12}(e)\circ f\cdot-\rho_{12}(f)\circ e\cdot)|_{S^{+}}\phantom{ccc}
\ee
$$
E^{*}\otimes S^{*}\otimes S\longrightarrow \Lambda^{2}E^{*}\otimes (S^{-})^{*}\otimes S^{+}\phantom{ccccccccccccccccccccccccccccccccccc}
$$
\be
\label{33}
\rho_{12}\longrightarrow \pi^{+}\circ(e\wedge f\rightarrow e\cdot\circ \rho_{12}(f)-f\cdot\circ\rho_{12}(e))|_{S_{-}}\phantom{ccc}
\ee
From \cite{ACDP} we know that 
$$\Cl_{0,8}\cong S^{*}\otimes S\cong \Lambda E\quad,\quad S^{+}\otimes S^{-}\cong E\oplus \Lambda^{3}E\quad,\quad S^{+}\otimes S^{+}\cong \bR\oplus\Lambda^{2}E\oplus\Lambda^{4}_{+}E$$
from which 
$$E^{*}\otimes S^{*}\otimes S\cong E^{*}\otimes\Lambda E$$
$$\Lambda^{2}E^{*}\otimes (S^{+})^{*}\otimes S^{+}\cong\Lambda^{2}E^{*}\oplus(\Lambda^{2}E^{*}\otimes\Lambda^{2}E)\oplus(\Lambda^{2}E^{*}\otimes\Lambda^{4}_{+}E)$$
$$\Lambda^{2}E^{*}\otimes (S^{-})^{*}\otimes S^{+}\cong (\Lambda^{2}E^{*}\otimes E)\oplus(\Lambda^{2}E^{*}\otimes \Lambda^{3}E)$$
and from the tables in \cite{OV} it is possible to further decompose these modules into irreducible components. In the next pages we draw two diagrams which exploit the kernel of the first two maps. On the left hand side there are all the irreducible modules of the domain $E^{*}\otimes\Lambda E$ divided in eight blocks corresponding to the decomposition $\Lambda E=\sum_{i=1}^{8}\Lambda^{i} E$. Similarly on the right for the codomain. According to the action of the map there are arrows which send a module on the left to the corresponding module on the right and the absence of arrows means that the module on the left is contained in the kernel of the map. We explicitly give a generator for all the irreducible modules of the domain:
$$
e_{1}^{*}\in R(\pi_{1})\subseteq E^{*}\otimes\bR
$$
$$
e_{1}^{*}\otimes e_{1}+\cdot\cdot\cdot +e_{8}^{*}\otimes e_{8}\in \bR\subseteq E^{*}\otimes E
$$
$$
e_{1}^{*}\otimes e_{1}\in R(2\pi_{1})\subseteq E^{*}\otimes E
$$
$$
e_{1}^{*}\otimes e_{2}-e_{2}^{*}\otimes e_{1}\in R(\pi_{2})\subseteq E^{*}\otimes E
$$
$$
e_{2}^{*}\otimes e_{1}\wedge e_{2}+\cdot\cdot\cdot +e_{8}^{*}\otimes e_{1}\wedge e_{8}\in R(\pi_{1})\subseteq E^{*}\otimes\Lambda^{2}E
$$
$$
e_{1}^{*}\otimes e_{2}\wedge e_{3}\in R(\pi_{1}+\pi_{2})\subseteq E^{*}\otimes\Lambda^{2}E
$$
$$
e_{1}^{*}\otimes e_{2}\wedge e_{3}+e_{3}^{*}\otimes e_{1}\wedge e_{2}+e_{2}^{*}\otimes e_{3}\wedge e_{1}\in R(\hat{\pi}_{3})\subseteq E^{*}\otimes\Lambda^{2}E
$$
$$
e_{3}^{*}\otimes e_{1}\wedge e_{2}\wedge e_{3}+\cdot\cdot\cdot +e_{8}^{*}\otimes e_{1}\wedge e_{2}\wedge e_{8}\in R(\pi_{2})\subseteq E^{*}\otimes\Lambda^{3}E
$$
$$
e_{1}^{*}\otimes e_{2}\wedge e_{3}\wedge e_{4}\in R(\pi_{1}+\hat{\pi}_{3})\subseteq E^{*}\otimes\Lambda^{3}E
$$
$$
e_{1}^{*}\otimes e_{2}\wedge e_{3}\wedge e_{4}+\cdot\cdot +e_{4}^{*}\otimes e_{1}\wedge e_{2}\wedge e_{3}+
$$
$$
\pm e_{5}^{*}\otimes e_{6}\wedge e_{7}\wedge e_{8}\pm\cdot\cdot \pm e_{8}^{*}\otimes e_{5}\wedge e_{6}\wedge e_{7}\wedge e_{8}\in\Lambda^{4}_{\pm}E\subseteq E^{*}\otimes \Lambda^{3}E
$$
$$
e_{1}^{*}\otimes e_{2}\wedge e_{3}\wedge e_{4}\wedge e_{5}+\cdot\cdot\cdot+e_{2}^{*}\otimes e_{3}\wedge e_{4}\wedge e_{5}\wedge e_{1}\in R(\hat{\pi}_{5})\subseteq E^{*}\otimes\Lambda^{4}E$$
$$
e_{4}^{*}\otimes e_{1}\wedge e_{2}\wedge e_{3}\wedge e_{4}+\cdot\cdot\cdot + e_{8}^{*}\otimes e_{1}\wedge e_{2}\wedge e_{3}\wedge e_{8}\in R(\hat{\pi}_{3})\subseteq E^{*}\otimes\Lambda^{4}E
$$
$$
e_{1}^{*}\otimes e_{2}\wedge e_{3}\wedge e_{4}\wedge e_{5}\in R(\pi_{1}+\hat{\pi}_{4})\subseteq E^{*}\otimes\Lambda^{4}E
$$
and the ones of the remaining irreducible submodules of $E^{*}\otimes\Lambda E\cong E^{*}\otimes\Cl_{0,8}$ are obtained through "Poincar\'e duality", \textit{i.e.} through the equivariant isomorphism
$
E^{*}\otimes\Cl_{0,8}\rightarrow E^{*}\otimes\Cl_{0,8}
$
given by
$
e^{*}\otimes c\rightarrow e^{*}\otimes c\omega_{0,8}
$
for every $c\in\Cl_{0,8}$.
The first diagram shows the kernel of the first map: notice that there is an irreducible module
which comes from a diagonal embedding in the domain given by
$$
R(\pi_{2})\longrightarrow R(\pi_{2})\oplus R(\pi_{2})\subseteq E^{*}\otimes(\Lambda^{3}E\oplus \Lambda^{7}E)
$$
The two generators of these two copies of $R(\pi_{2})$ map to the same generator of $R(\pi_{2})\subseteq \Lambda^{2}E^{*}\otimes\Lambda^{2}E$ and this shows how to construct the embedding. From the second diagram we see that this module is in the kernel of the second map and, by direct calculation, also of the third map. This module is then made up of suitable morphisms. Then we study the second map and we define diagonal embeddings
as before when needed. Similar considerations and an explicit calculation of the image of the various generators imply the result of the proposition. More explicitly the module of suitable morphisms is made up of
$$
R(\pi_{1})\subseteq E^{*}\otimes (\Lambda^{0}E\oplus\Lambda^{8}E)
$$
$$
\bR\oplus R(2\pi_{1})\subseteq E^{*}\otimes E
$$
$$
R(\pi_{1}+\pi_{2})\oplus R(\hat{\pi}_{3})\oplus R(\pi_{1})\subseteq E^{*}\otimes (\Lambda^{2}E\oplus\Lambda^{6}E)
$$
$$
R(\pi_{2})\subseteq E^{*}\otimes (\Lambda^{3}E\oplus\Lambda^{7}E)
$$ 
$$
R(2\pi_{4})\subseteq E^{*}\otimes\Lambda^{5}E\qquad.
$$
\end{proof}
\newpage
Diagram of the map (\ref{11}):
$$
\xymatrix@C=1pc@R=0pc{
R(\pi_{1}) & & & & & & & & & & &\\
& & & & & & & & & & &\\
& & & & & & & & & & &\\
& & & & & & & & & & &\\
\bR & & & & & & & & & & & \\
R(2\pi_{1}) & & & & & & & & & & & \\
R(\pi_{2}) \ar@/_/[rrrrrrrrrrr]& & & & & & & & & & & R(\pi_{2})\\
& & & & & & & & & & &\\
& & & & & & & & & & &\\
& & & & & & & & & & &\\
R(\pi_{1})& & & & & & & & & & & \\
R(\pi_{1}+\pi_{2}) & & & & & & & & & & & \\
R(\hat{\pi}_{3}) & & & & & & & & & & & \\
& & & & & & & & & & & \\
& & & & & & & & & & & R(\pi_{1}+\hat{\pi}_{3})\\
& & & & & & & & & & & \bR\\
R(\pi_{1}+\hat{\pi}_{3})\ar@/^/[dddddddddddddrrrrrrrrrrr] & & & & & & & & & & & R(2\pi_{1})\\
R(\pi_{2})\ar@/_/[rrrrrrrrrrr]& & & & & & & & & & & R(\pi_{2})\\
R(2\pi_{3})\ar@/_/[rrrrrrrrrrr]& & & & & & & & & & & R(2\pi_{3})\\
R(2\pi_{4})\ar@/_/[rrrrrrrrrrr]& & & & & & & & & & & R(2\pi_{4})\\
& & & & & & & & & & &R(2\pi_{2})\\
& & & & & & & & & & &\\
& & & & & & & & & & &\\
R(\hat{\pi}_{3}) & & & & & & & & & & & \\
R(\hat{\pi}_{5}) & & & & & & & & & & & \\
R(\pi_{1}+\hat{\pi}_{4}) & & & & & & & & & & & \\
& & & & & & & & & & &\\
& & & & & & & & & & &\\
& & & & & & & & & & &\\
R(\pi_{1}+\hat{\pi}_{3})\ar@(r,ul)[uuuuuuuuuuuuuuurrrrrrrrrrr] & & & & & & & & & & & R(\pi_{1}+\hat{\pi}_{3})\\
R(\pi_{2})\ar@/_/[rrrrrrrrrrr]& & & & & & & & & & & R(\pi_{2})\\
R(2\pi_{3})\ar@/_/[rrrrrrrrrrr]& & & & & & & & & & & R(2\pi_{3})\\
R(2\pi_{4})& & & & & & & & & & & R(2\pi_{3}+\pi_{2})\\
& & & & & & & & & & & \\
& & & & & & & & & & & \\
& & & & & & & & & & & \\
R(\pi_{1})& & & & & & & & & & &  \\
R(\pi_{1}+\pi_{2}) & & & & & & & & & & &\\
R(\hat{\pi}_{3})& & & & & & & & & & & \\
& & & & & & & & & & & \\
& & & & & & & & & & & \\
& & & & & & & & & & &  \\
\bR \ar@(r,ul)[uuuuuuuuuuuuuuuuuuuuuuuuuuurrrrrrrrrrr]&&&&&&&&&&& \\
R(2\pi_{1})\ar@(r,ul)[uuuuuuuuuuuuuuuuuuuuuuuuuuurrrrrrrrrrr] & & & & & & & & & & & \\
R(\pi_{2})\ar@(r,ul)[uuuuuuuuuuuuuuuuuuuuuuuuuuurrrrrrrrrrr]& & & & & & & & & & & \\
& & & & & & & & & & & \\
& & & & & & & & & & & \\
& & & & & & & & & & & \\
R(\pi_{1})& & & & & & & & & & & \\
}$$
\newpage
Diagram of the map (\ref{22}):
$$
\xymatrix@C=1pc@R=0pc{
R(\pi_{1})\ar@(l,dl)[ddddddddddddrrrrrrrrrrr] & & & & & & & & & & &\\
& & & & & & & & & & &\\
& & & & & & & & & & &\\
& & & & & & & & & & &\\
\bR & & & & & & & & & & & \\
R(2\pi_{1}) & & & & & & & & & & & \\
R(\pi_{2}) & & & & & & & & & & & \\
& & & & & & & & & & &\\
& & & & & & & & & & &\\
& & & & & & & & & & &\\
R(\pi_{1}+\pi_{2})\ar@/_/[rrrrrrrrrrr]\ar@(r,d)[ddddddddddddddddddddddddddddrrrrrrrrrrr] & & & & & & & & & & & R(\pi_{1}+\pi_{2})\\
R(\hat{\pi}_{3}) \ar@/_/[rrrrrrrrrrr]\ar@(u,u)[dddddddddddddddddddddddddrrrrrrrrrrr] & & & & & & & & & & &R(\hat{\pi_{3}}) \\
R(\pi_{1})\ar@/_/[dddddddddddddddddddddddddrrrrrrrrrrr]\ar@/_/ [rrrrrrrrrrr]& & & & & & & & & & & R(\pi_{1})\\
& & & & & & & & & & & \\
& & & & & & & & & & &\\
& & & & & & & & & & &\\
R(\pi_{2}) & & & & & & & & & & & \\
R(2\pi_{3}) & & & & & & & & & & & \\
R(2\pi_{4}) & & & & & & & & & & & \\
R(\pi_{1}+\hat{\pi}_{3}) & & & & & & & & & & & \\
& & & & & & & & & & &\\
& & & & & & & & & & &\\
& & & & & & & & & & &\\
R(\hat{\pi}_{3}) \ar@(u,u)[dddddddddddddrrrrrrrrrrr]& & & & & & & & & & & \\
R(\hat{\pi}_{5})\ar@(u,l)[dddddddddddddddrrrrrrrrrrr] & & & & & & & & & & & \\
R(\pi_{1}+\hat{\pi}_{4})\ar@(u,l)[dddddddddddddddrrrrrrrrrrr] & & & & & & & & & & & \\
& & & & & & & & & & &\\
& & & & & & & & & & &\\
& & & & & & & & & & &\\
R(\pi_{1}+\hat{\pi}_{3})  & & & & & & & & & & & \\
R(\pi_{2})& & & & & & & & & & &\\
R(2\pi_{3})& & & & & & & & & & &\\
R(2\pi_{4})& & & & & & & & & & &\\
& & & & & & & & & & & \\
& & & & & & & & & & & \\
& & & & & & & & & & & \\
R(\hat{\pi}_{3}) \ar@(r,ul)[uuuuuuuuuuuuuuuuuuuuuuuuurrrrrrrrrrr] \ar@/_/[rrrrrrrrrrr]& & & & & & & & & & & R(\hat{\pi}_{3})\\
R(\pi_{1}) \ar@(r,d)[uuuuuuuuuuuuuuuuuuuuuuuuurrrrrrrrrrr] \ar@/_/[rrrrrrrrrrr]& & & & & & & & & & &  R(\pi_{1})\\
R(\pi_{1}+\pi_{2}) \ar@(r,ul)[uuuuuuuuuuuuuuuuuuuuuuuuuuuurrrrrrrrrrr] \ar@/_/[rrrrrrrrrrr]& & & & & & & & & & & R(\pi_{1}+\pi_{2})\\
& & & & & & & & & & & R(\hat{\pi}_{5})\\
& & & & & & & & & & & R(\pi_{1}+\hat{\pi}_{4})\\
& & & & & & & & & & & R(\pi_{2}+\hat{\pi}_{3}) \\
\bR & & & & & & & & & & & \\
R(2\pi_{1}) & & & & & & & & & & & \\
R(\pi_{2}) & & & & & & & & & & & \\
& & & & & & & & & & & \\
& & & & & & & & & & & \\
& & & & & & & & & & & \\
R(\pi_{1})\ar@(r,d)[uuuuuuuuuuuuuuuuuuuuuuuuuuuuuuuuuuuurrrrrrrrrrr]& & & & & & & & & & & \\
}$$
To complete the analysis of all possible non-zero charge superizations of the Heisenberg algebra in dimension $\dim E=n-1=8$ we need to examine all possible brackets between odd elements. In this case we have the following proposition.
\begin{proposition}
If $\dim E=n-1=8$, every non-zero charge superization of the Heisenberg algebra $\gheis_{b}=E^{*}+E+\bR\mathrm{p}$ is odd commutative.
\end{proposition}
\begin{proof}
The bracket between two odd elements $Q,\tilde{Q}\in S$ is denoted by
\be
\label{bispinore}
[Q,\tilde{Q}]:=\mathrm{p}(Q,\tilde{Q})_{\mathrm{p}}\mathrm{p}+e_{i}(Q,\tilde{Q})e_{i}+e_{i}^{\flat}(Q,\tilde{Q})e_{i}^{\flat}
\ee
where $\mathrm{p}(\cdot,\cdot), e_{i}(\cdot,\cdot), e_{i}^{\flat}(\cdot,\cdot)\in\Bil_{\bR}(S)$ are symmetric bilinear forms on $S$. Decompose every symmetric bilinear form on $S$ into three parts
(\ref{r--}), (\ref{r-+}), (\ref{r++}) as before. The space $\Bil_{\bR}(S_{0,8})^{\so(0,8)}$ has an admissible basis (see \cite{AC}) given by the bilinear forms $h$ and $h_{E}:=h(E\cdot,\cdot)$ whose invariants are given by $$(\tau,\sigma,\i)(h)=(+1,+1,+1)\qquad,\qquad(\tau,\sigma,\i)(h_{E})=(-1,+1,+1)\qquad.$$
The bilinear form $h$ is a positively defined scalar product on $S_{0,8}$ (see \cite{LM}) satisfying the following property:
$$
h(e_{i}s,e_{j}s)=h(e_{j}e_{i}s,s)=-h(e_{i}e_{j}s,s)=-h(e_{j}s,e_{i}s)=-h(e_{i}s,e_{j}s)=0
$$
for every $1\leq i\lneq j\leq 8$ and $s,t\in S_{0,8}$. We have to write the condition that the bracket (\ref{bispinore}) is $\gheis_{b}$-invariant and that it satisfies the odd Jacobi identity
\be
\label{000}
[Q,[Q,Q]]=0
\ee 
The condition of $E^{*}$-invariance is equivalent to
\be
\label{un}
e_{j}(Q,\tilde{Q})=\sqrt{2}\p(\left(
\begin{array}{c}
e_{j}\cdot Q_{+} \\
0 \\
\end{array}
\right),
\left(
\begin{array}{c}
\tilde{Q}_{-}\\
\tilde{Q}_{+} \\
\end{array}
\right))+\sqrt{2}
\p(\left(
\begin{array}{c}
Q_{-} \\
Q_{+} \\
\end{array}
\right),
\left(
\begin{array}{c}
e_{j}\cdot\tilde{Q}_{+} \\
0 \\
\end{array}
\right))
\ee
\be
\label{du}
0=e_{i}^{\flat}(\left(
\begin{array}{c}
e_{j}\cdot Q_{+} \\
0 \\
\end{array}
\right),
\left(
\begin{array}{c}
\tilde{Q}_{-}\\
\tilde{Q}_{+} \\
\end{array}
\right))+
e_{i}^{\flat}(\left(
\begin{array}{c}
Q_{-} \\
Q_{+} \\
\end{array}
\right),
\left(
\begin{array}{c}
e_{j}\cdot\tilde{Q}_{+} \\
0 \\
\end{array}
\right))=0
\ee
\be
\label{tr}
0=e_{i}(\left(
\begin{array}{c}
e_{j}\cdot Q_{+} \\
0 \\
\end{array}
\right),
\left(
\begin{array}{c}
\tilde{Q}_{-}\\
\tilde{Q}_{+} \\
\end{array}
\right))+
e_{i}(\left(
\begin{array}{c}
Q_{-} \\
Q_{+} \\
\end{array}
\right),
\left(
\begin{array}{c}
e_{j}\cdot\tilde{Q}_{+} \\
0 \\
\end{array}
\right))=0
\ee
for every $1\leq i,j\leq 8$. Equations (\ref{du}), (\ref{tr}) and Lemma \ref{MM} imply that 
$$
(e_{i}^{\flat})_{--}=(e_{i})_{--}=0
$$
and that $(e_{i}^{\flat})_{-+}$ and $(e_{i})_{-+}$ must be proportional to $h_{E}:=h(E\cdot,\cdot)\in\Bil_{\bR}(S_{0,8})$. This assertion {\bf strongly depends} on the fact that, when $\dim E=8$, there is, up to scalar, only one admissible bilinear form on $S_{0,8}$ whose invariants are such that $\tau\sigma=-1$, namely $h_{E}$. Equations (\ref{un}) are equivalent to 
$$
(e_{i})_{-+}(Q_{-},\tilde{Q}_{+})=\sqrt{2}\p_{--}(Q_{-},e_{i}\cdot \tilde{Q}_{+})
$$
$$
(e_{i})_{++}(Q_{+},\tilde{Q_{+}})=\sqrt{2}(\p_{-+}(e_{i}\cdot Q_{+} ,\tilde{Q}_{+})+\p_{-+}(e_{i}\cdot \tilde{Q}_{+},Q_{+}))
$$
and it is easy to see that $(e_{i})_{-+}=\p_{--}=0$. The condition of $\p$-invariance is then equivalent to
$$
\p_{-+}(\pi^{+}Q_{+},\tilde{Q}_{+})+\p_{-+}(\pi^{+}\tilde{Q}_{+},Q_{+})=0
$$
$$
(e_{i}^{\flat})_{-+}(\pi^{+}Q_{+},\tilde{Q}_{+})+(e_{i}^{\flat})_{-+}(\pi^{+}\tilde{Q}_{+},Q_{+})=0
$$
and it is easy to see that $(e_{i}^{\flat})_{-+}=0$ (recall again that $h_{E}$ is, up to scalar, the only invariant form with invariants $(\tau,\sigma)$ such that $\tau\sigma=-1$). The non-linear condition (\ref{000}) can then be translated into two equations
\be
\sum_{i}(e_{i})_{++}(Q_{+},Q_{+})(h_{3}(e_{i})+2h_{1}e_{i}\pi^{+})Q_{+}=0
\ee
$$
2(2\sqrt{2}h_{1}\p(Q,Q)\pi^{+}+\sqrt{2}b_{i}(e_{i}^{\flat})_{++}(Q_{+},Q_{+})e_{i}+
$$
\be
\label{pau}
(e_{i})_{++}(Q_{+},Q_{+})\rho_{12}(e_{i}))\phantom{c}Q_{+}\phantom{c}+
(e_{i})_{++}(h_{3}(e_{i})-2e_{i}\pi^{-})\phantom{c}Q_{-}=0
\ee
The first equation can be split according to $(\ref{subito})$:
\be
\label{splitu}
0=\sum_{i}(e_{i})_{++}(Q_{+},Q_{+})h_{3}(e_{i})Q_{+}^{+}\in S_{0,8}^{+}
\ee
\be
\label{splitd}
0=\sum_{i}(e_{i})_{++}(Q_{+},Q_{+})h_{3}(e_{i})Q_{+}^{-}+\sum_{i}(e_{i})_{++}(Q_{+},Q_{+})2h_{1}e_{i}Q_{+}^{+}=0\in S_{0,8}^{-}
\ee
Taking the scalar product of (\ref{splitu}) with $Q_{+}^{+}$, we get that
$$0=\sum_{i}(e_{i})_{++}(Q_{+},Q_{+})h_{3}(e_{i})$$
and then, taking the scalar product of (\ref{splitd}) with $e_{j}Q_{+}^{+}$ for a fixed $1\leq j\leq n-1$, we get that
$$
0=\sum_{i}(e_{i})_{++}(Q_{+},Q_{+})2h_{1}h(e_{i}Q_{+}^{+},e_{j}Q_{+}^{+})=(e_{j})_{++}(Q_{+},Q_{+})2h_{1}h(e_{j}Q_{+}^{+},e_{j}Q_{+}^{+})
$$
which is equivalent to $(e_{j})_{++}=0$. Summarizing we get that
$$
(e_{i})_{--}=(e_{i})_{-+}=(e_{i})_{++}=(e_{i}^{\flat})_{--}=(e_{i}^{\flat})_{-+}=\p_{--}=0
$$
$$
\p_{-+}(\pi^{+}Q_{+},\tilde{Q}_{+})+\p_{-+}(\pi^{+}\tilde{Q}_{+},Q_{+})=0=\p_{-+}(e_{i}\cdot Q_{+},\tilde{Q}_{+})+\p_{-+}(e_{i}\cdot \tilde{Q}_{+},Q_{+})
$$
$$
(2h_{1}\p(Q,Q)\pi^{+}+b_{i}(e_{i}^{\flat})_{++}(Q_{+},Q_{+})e_{i})\phantom{c}Q_{+}=0
$$
In particular, setting $Q_{-}=0$, last equation implies that
$$
(2h_{1}\p_{++}(Q_{+},Q_{+})\pi^{+}+b_{i}(e_{i}^{\flat})_{++}(Q_{+},Q_{+})e_{i})\phantom{c}Q_{+}=0
$$
and this equation can be split according to $(\ref{subito})$. The term belonging to $S_{0,8}^{-}$ is given by
$$
b_{i}(e_{i}^{\flat})_{++}(Q_{+},Q_{+})e_{i}\cdot Q_{+}^{+}=0
$$
and taking the scalar product of it with $e_{j}Q_{+}^{+}$, for a fixed $1\leq j\leq n-1$, we get that
$$
0=b_{i}(e_{i}^{\flat})_{++}(Q_{+},Q_{+})h(e_{i}\cdot Q_{+}^{+},e_{j}Q_{+}^{+})=b_{j}(e_{j}^{\flat})_{++}(Q_{+},Q_{+})h(e_{j}\cdot Q_{+}^{+},e_{j}Q_{+}^{+})
$$
from which $(e_{j}^{\flat})_{++}=0$. Taking into account the obtained results, equation (\ref{pau}) implies that $\p=0$. The proposition is thus proved. 
\end{proof}
There was no need to consider the condition of $E$-invariance for the bracket between odd elements. An important ingredient of the proof was indeed the fact that there is only one, up to scalar, admissible bilinear form on $S_{0,8}$ whose invariants are such that $\tau\sigma=-1$. This implies that $[S,S]\subseteq E^{*}+\bR\p$ and drastically simplifies the equations.
\section{Spin representation of the Cahen-Wallach algebra}
\label{dan}
\setcounter{equation}{0}
This section considers extension of spin representations of the Heisenberg algebra $\gheis_{b}=E^{*}+E+\bR\p$ to representations
\be
\label{dor}
\rho:\gg_{b}\rightarrow\ggl_{\bR}(S)
\ee
of the CW algebra $$\gg_{b}=(E^{*}+E+\bR\p)+\bR\q=\gheis_{b}+\bR\q$$ defined by the symmetric endomorphism $B\in\End_{\bR}(E)$. To this end we specify the image 
$$
\rho(\q)=\left(
\begin{array}{cc}
\q_{11} & \q_{12} \\
\q_{21} & \q_{22} \\
\end{array}
\right)
$$
of the extra element $\q\in\gg_{b}$ with respect to the decomposition (\ref{dec}). Firstly Lemmas \ref{ne}, \ref{reg} give some necessary conditions so that a spin representation of the Heisenberg algebra can be extended to a representation of the CW algebra. In the zero charge case, Theorem \ref{iiii} shows that spin representations (\ref{dor}) are determined by solutions of a quadratic equation in the Clifford algebra $\Cl_{0,n-1}$. 
The complete understanding of this equation is far from being reached.
\subsection{Zero charge representation}\hfill
\begin{lemma}
\label{ne}
For any zero charge spin representation of a CW algebra the endomorphisms $\q_{21}\in\End_{\bR}(S_{0,n-1})$ and $\rho_{11}\in\Hom_{\bR}(E,\cC_{0,n-1})$ are zero and 
\be
\label{ciao}
\rho_{12}(e)=\sqrt{2}(e\cdot\circ\phantom{c}\q_{22}-\q_{11}\circ e\cdot)
\ee
\end{lemma}
\begin{proof} 
By Theorem \ref{0c}, any suitable pair defines a zero charge representation (\ref{common}) of $\gheis$. For every $e_{i}\in E$, equation $[e_{i}^{\flat},\q]=b_{i}e_{i}$ implies that
$$
\left(
\begin{array}{cc}
0 & \sqrt{2} e_{i} \\
0 & 0 \\
\end{array}
\right)
\left(
\begin{array}{cc}
\q_{11} & \q_{12} \\
\q_{21} & \q_{22} \\
\end{array}
\right)-
\left(
\begin{array}{cc}
\q_{11} & \q_{12} \\
\q_{21} & \q_{22} \\
\end{array}
\right)
\left(
\begin{array}{cc}
0 & \sqrt{2} e_{i} \\
0 & 0 \\
\end{array}
\right)=
\left(
\begin{array}{cc}
\rho_{11}(e_{i}) & \rho_{12}(e_{i}) \\
0 & \overline{\rho_{11}}(e_{i}) \\
\end{array}
\right)
$$
\textit{i.e.}
$$
\sqrt{2}
\left(
\begin{array}{cc}
e_{i}\cdot\circ \phantom{c}\q_{21} &  e_{i}\cdot\circ\phantom{c} \q_{22}-\q_{11} \circ e_{i}\cdot\\
0 & -\q_{21} \circ e_{i}\cdot \\
\end{array}
\right)
=\left(
\begin{array}{cc}
\rho_{11}(e_{i}) & \rho_{12}(e_{i}) \\
0 & \overline{\rho_{11}}(e_{i}) \\
\end{array}
\right)\quad.
$$
This equation is equivalent to
$$
\left\{\begin{matrix} \sqrt{2}\q_{21}=e_{i}\cdot\circ\phantom{c}\rho_{11}(e_{i})=-\overline{\rho_{11}}(e_{i})\circ e_{i}\cdot\phantom{cc}\\
\rho_{12}(e_{i})=\sqrt{2}(e_{i}\cdot\circ\phantom{c} \q_{22}-\q_{11} \circ e_{i}\cdot)
\end{matrix}\right.
$$
and then
$$
\sqrt{2}\q_{21}=e_{i}\cdot\circ\phantom{c}\rho_{11}(e_{i})=\overline{\rho_{11}}(e_{i})\circ e_{i}\cdot=0
$$
from which $\rho_{11}=0$ and $\q_{21}=0$.
\end{proof}
The following main theorem describe all zero charge spin representations of a CW algebra in terms of triple of elements of the Clifford algebra $\Cl_{0,n-1}$ such that the first two satisfy the following quadratic equation. 
\begin{definition}
\rm{
A pair $(c_{1},c_{2})\in(\Cl_{0,n-1})^{2}$ satisfy the {\bf quadratic Clifford equation} if
\be
\label{QUAD}
i)\pc\pc(L_{c_{1}}-R_{c_{2}})^{2}E\subseteq E\phantom{ccc},\phantom{ccc}ii)\pc\pc(L_{c_{1}}-R_{c_{2}})^{2}|_{E}\equiv -B
\ee
where $L_{c}\in\End_{\bR}(\Cl_{0,n-1})$ and $R_{c}\in\End_{\bR}(\Cl_{0,n-1})$ are the operators of left and right multiplication by $c\in\Cl_{0,n-1}$.
}\end{definition}
\begin{theorem}
\label{iiii}
Let $$(c_{1},c_{2})\in(\Cl_{0,n-1})^{2}$$ be a solution of the {\bf quadratic Clifford equation} (\ref{QUAD})
and $c_{3}\in\Cl_{0,n-1}$. Every such a triple defines a zero charge spin representation of a $(n+1)$-dimensional $CW$ algebra $$\gg_{b}=\gh+\gm=E^{*}+(E+\bR\p+\bR\q)$$ given by
$$
\phantom{c}\rho(e^{\flat})=\sqrt{2}\begin{pmatrix} 0 & Be \\ 0 & 0 \end{pmatrix}\phantom{cccc}\quad,\quad
\rho(\mathrm{p})=\begin{pmatrix} 0 & 0 \\ 0 & 0 \end{pmatrix}
$$
$$
\rho(e)=\sqrt{2}\begin{pmatrix} 0 & e\cdot\circ\phantom{c} c_{2}-c_{1}\circ\phantom{c}e\cdot\\ 0 & 0 \end{pmatrix}\quad,\quad
\rho(\q)=\begin{pmatrix} c_{1} & c_{3} \\ 0 & c_{2} \end{pmatrix}\qquad.
$$
Every zero charge spin representation is defined by such a triple and two triples define the same representation if they differ by elements of the kernel (if it exists) of the Euclidean spin representation $\rho_{0,n-1}:\Cl_{0,n-1}\rightarrow\End_{\bR}(S_{0,n-1})$.
\end{theorem}
\begin{proof}
Lemma \ref{ne} and equation $[\q,e]=e^{\flat}$ imply that
\be
\label{def}
\left(
\begin{array}{cc}
0 &  \q_{11}\circ\rho_{12}(e)-\rho_{12}(e)\circ\q_{22}\\
0 & 0 \\
\end{array}
\right)
=\sqrt{2}\left(
\begin{array}{cc}
0 & Be \\
0 & 0 \\
\end{array}
\right)
\ee
and substituting equation (\ref{ciao}) into (\ref{def}) we get the following quadratic condition
$$
\q_{11}\circ e\cdot\circ\phantom{c}\q_{22}-(\q_{11})^{2}\circ e\cdot- e\cdot\circ\phantom{c}(q_{22})^{2}+\q_{11}\circ e\cdot\circ\phantom{c}\q_{22}=Be\cdot\qquad\qquad.
$$
This reduces to equation (\ref{QUAD}) if we remark that any zero charge spin representation (\ref{dor}) of a CW algebra satisfies, by assumption, $\rho(\gm)\subseteq\rho_{\spin}(\Cl(\gm))$. The theorem is thus proved.
\end{proof}
Note that the quadratic Clifford equation makes sense in every Clifford algebra. If the symmetric endomorphism $B\in\End_{\bR}(E)$ is not specified, the second condition of (\ref{QUAD}) is equivalent to the condition that the endomorphism $(L_{c_{1}}-R_{c_{2}})^{2}|_{E}$ is diagonalizable.
Indeed equation (\ref{QUAD}) is invariant under the adjoint action of the Spin group on the Clifford algebra. \\
We do not discuss the general solution of (\ref{QUAD}), which looks very complicated, but only consider three simple types of solutions.\\
\\
{\bf Linear solutions}: This means that the pair $(c_{1},c_{2})\in(\Cl_{0,n-1})^{2}$ satisfies $(L_{c_{1}}-R_{c_{2}})E\subseteq E$. For example, choose $$c_{1}=r_{1}\Id+s_{1}\omega_{0,n-1}\qquad,\qquad c_{2}=r_{2}\Id-s_{1}\omega_{0,n-1}$$ where $r_{1}\neq r_{2}$ are two real numbers. \\
\\
{\bf Half-zero solutions}: The name is self-explanatory. For example, choose $c_{1}=0$ and $c_{2}$ to be, up to a positive scalar, any involution or complex structure. In general these solutions are not linear. \\
\\
{\bf Supergravity solutions}:  These are solutions which consist of proportional decomposable elements of the Clifford algebra, \textit{i.e.}  $$c_{1}=r_{1}e_{I}\qquad,\qquad c_{2}=r_{2}e_{I}$$ 
where $r_{1}\neq\pm r_{2}$ are two real numbers and $e_{I}=e_{i_{1}}\cdot\cdot\cdot e_{i_{|I|}}\in\Cl_{0,n-1}$ is the decomposable form of degree $|I|$ determined by the multi-index $I=(i_{1},...,i_{|I|})$. In general these solutions are not linear.  \\
\\
The last type of solutions is the most interesting one. Solutions of this type appear in $11$-dimensional supergravity (see \cite{FP}), hence the name.\\
One can easily check that linear and half-zero solutions correspond to a scalar endomorphism $B\in\End_{\bR}(E)$, while supergravity solutions correspond to an endomorphism with two different eigenvalues. This could be an indication that the quadratic Clifford equation is solvable only for some symmetric matrix $B$. It would be very interesting to have a deeper understanding of (\ref{QUAD}).  
\subsection{Non-zero charge representation}\hfill\newline\\
This subsection considers the case when the representation of the CW algebra has non-zero charge. This can happen only when semi-spinors exist, \textit{i.e.} when $\dim E=n-1\equiv 0,4,6,7$ (mod $8$). We first prove some general results and then specialize ourselves to the case $\dim E=n-1=8$. 
In this case, 
the theorem \ref{check} proves that spin representations of the CW algebra $\gg_{b}$ exist only when the symmetric endomorphism $B$ is scalar. We describe a solution and we then check that the obtained formula gives solution in any dimension.
\subsubsection{General case}\hfill\newline\\
Assume that the homomorphisms $\rho_{11},\rho_{12},\rho_{22}\in\Hom_{\bR}(E,\End_{\bR}(S_{0,n-1}))$ give a non-zero charge spin representation of the Heisenberg algebra $\gheis_{b}=E^{*}+E+\bR\p$. Recall that $\rho_{11}$ and $\rho_{22}$ are described in terms of elements $h_{1},h_{2},h_{3},h_{4}$ (which are either elements of the even Schur algebra $\cC^{\circ}_{0,n-1}$ or linear maps from $E$ to $\cC_{0,n-1}^{\circ}$) and that $\rho_{12}$ is a suitable map (see Theorems \ref{0p}, \ref{3p}, \ref{2p}, \ref{3d}). An extension of the representation 
\be
\label{grande}
\rho:\gheis_{b}\rightarrow\ggl_{\bR}(S)
\ee
to a representation of a CW algebra (\ref{dor}) can be reduced to finding solutions $\q_{11},\q_{12},\q_{21},\q_{22}\in\End_{\bR}(S_{0,n-1})$ to the following system of equations 
\be
\label{I}
\left\{\begin{matrix} \sqrt{2}\q_{21}=e_{i}\cdot\circ\phantom{c}\rho_{11}(e_{i})=-\rho_{22}(e_{i})\circ e_{i}\cdot\phantom{cc}\\
\rho_{12}(e_{i})=\sqrt{2}(e_{i}\cdot\circ\phantom{c} \q_{22}-\q_{11} \circ e_{i}\cdot)
\end{matrix}\right.
\ee
and
\be
\label{II}
\left(
\begin{array}{cc}
[\q_{11},\rho_{11}(e_{i})] &  \q_{11}\circ\rho_{12}(e_{i})+\q_{12}\circ\rho_{22}(e_{i})\\
-\rho_{12}(e_{i})\circ\q_{21}& -\rho_{11}(e_{i})\circ\phantom{c}\q_{12}-\rho_{12}(e_{i})\circ\q_{22}\\
& \\
\q_{21}\circ\rho_{11}(e_{i})& [\q_{22},\rho_{22}(e_{i})]\\
-\rho_{22}(e_{i})\circ\phantom{c}\q_{21} & +\q_{21}\circ\rho_{12}(e_{i}) \\
\end{array}
\right)
=b_{i}\sqrt{2}\left(
\begin{array}{cc}
0 & e_{i} \\
0 & 0 \\
\end{array}
\right)
\ee
\begin{lemma}
\label{reg}
Assume that a representation (\ref{grande}) of the Heisenberg algebra admits an extension to a representation of the CW algebra. Then
\begin{enumerate}
\item[i)]
If $n-1\equiv 0\pc (mod\pc 8)$, then
$h_{3}\in\Hom_{\bR}(E,\cC_{0,n-1}^{\circ})$
is zero and $$2\sqrt{2}\q_{21}=-h_{1}(\Id\mp E)\quad,$$
\item[ii)]
If $n-1\equiv 4\pc (mod\pc 8)$, then
$h_{3},h_{4}\in\Hom_{\bR}(E,\cC_{0,n-1}^{\circ})$
are zero and $$2\sqrt{2}\q_{21}=-h_{1}(\Id\mp E)\quad,$$
\item[iii)]
If $n-1\equiv 6\pc (mod\pc 8)$, then
$h_{3},h_{4}\in\Hom_{\bR}(E,\cC_{0,n-1}^{\circ})$
are zero and $$2\sqrt{2}\q_{21}=-h_{1}+h_{2}EI\quad,$$
\item[iv)]
If $n-1\equiv 7\pc (mod\pc 8)$, then
$h_{3}\in\Hom_{\bR}(E,\cC_{0,n-1}^{\circ})$
is zero and $$2\sqrt{2}\q_{21}=\pm|h_{4}|-h_{4}J\quad.$$
\end{enumerate}
\end{lemma}
\begin{proof}
Consider the first equation of system (\ref{I}) and the $(2,1)$-entry of the matrix (\ref{II}). We prove the lemma only in the case $n-1\equiv 6\pc (\text{mod}\pc 8)$, the other cases are similar. First equation of system $(\ref{I})$ implies that
$$
h_{3}(e_{i})\Id+h_{4}(e_{i})EI=0
$$
\textit{i.e.} $h_{3}=h_{4}=0$ and then
$
2\sqrt{2}\q_{21}=-h_{1}+h_{2}EI
$.
The equation corresponding to the $(2,1)$-entry of the matrix of (\ref{II}) is satisfied. Indeed this entry is proportional to
$
(h_{1}^{2}+h_{2}^{2})e_{i}-2[h_{1},h_{2}]e_{i}EI
$, 
which is zero due to Lemma \ref{quat2}. 
\end{proof}
Making use of Lemma \ref{reg}, re-write the main equations (\ref{I}) and (\ref{II}) as 
\be
\label{n1}
\rho_{12}(e)=\sqrt{2}(e\cdot\circ\phantom{c}\q_{22}-\q_{11}\circ e\cdot)
\ee
\be
\label{n2}
[\q_{11},\rho_{11}(e)]=\rho_{12}(e)\circ\phantom{c}\q_{21}
\ee
\be
\label{n3}
[\q_{22},\rho_{22}(e)]=-\q_{21}\circ\rho_{12}(e)
\ee
\be
\label{n4}
(\q_{11}\circ\rho_{12}(e)-\rho_{12}(e)\circ\phantom{c}\q_{22})+(\q_{12}\circ\rho_{22}(e)-\rho_{11}(e)\circ\phantom{c}\q_{12})=\sqrt{2}Be
\ee
\subsubsection{Case $\dim E=n-1=8$}\hfill\newline\\
Let us restrict ourselves to the case $\dim E=n-1=8$.
Theorem \ref{0p} and Lemma \ref{reg} show that non-zero charge spin representations of the Heisenberg algebra on the spin module $S_{1,9}$ are given by
$$
\phantom{cccccccccccccc}\rho(e^{\flat})=\sqrt 2\begin{pmatrix} 0 & Be \\ 0 & 0 \end{pmatrix}\phantom{c}
,\phantom{c}
\rho(e)=\begin{pmatrix} -h_{1}e\circ\pi^{-} & \rho_{12}(e) \\ 0 & +h_{1}e\circ\pi^{+}  \end{pmatrix}
$$
$$
\rho(\mathrm{p})=2\sqrt{2}h_{1}\begin{pmatrix} 0 & \pi^{+} \\ 0 & 0 \end{pmatrix}
\phantom{c},\phantom{c}
\rho(\q)=\left(
\begin{array}{cc}
\q_{11} & \q_{12} \\
\q_{21} & \q_{22} \\
\end{array}
\right)
$$
where $\rho_{12}\in\Hom_{\bR}(\bR^{0,8},\End_{\bR}(S_{0,8}))$ is an element of the $\so(0,8)$-module which consists of suitable maps (see Proposition \ref{S}) and $\sqrt{2}\q_{21}=-h_{1}\pi^{-}$ (see Lemma \ref{reg}).
\begin{theorem}
\label{check}
A CW algebra $\gg_{b}=\gh+\gm=E^{*}+(E+\bR\p+\bR\q)$ with $\dim E=n-1=8$ admits non-zero charge spin representations if and only if $B=\pm\Id$. In this case a representation is given by
$$
\phantom{ccccccccccccc}\rho(e^{\flat})=\sqrt 2\begin{pmatrix} 0 & Be \\ 0 & 0 \end{pmatrix}\phantom{c}
,\phantom{c}
\rho(e)=\sqrt{2}\begin{pmatrix} -e\circ\pi^{-} & 0 \\ 0 & +e\circ\pi^{+}  \end{pmatrix}
$$
\be
\label{formuloneoneone}
\phantom{cccccccc}\rho(\mathrm{p})=4\begin{pmatrix} 0 & \pi^{+} \\ 0 & 0 \end{pmatrix}
\phantom{c},\phantom{c}
\rho(\q)=\left(
\begin{array}{cc}
0 & \pm\Id \\
-\pi^{-} & 0 \\
\end{array}
\right)
\ee
\end{theorem}
\begin{proof}
Every endomorphism of the spin module $S_{0,8}$ is decomposed according to (\ref{subito}). For example, $\q_{11}$ is identified with a $2\times 2$ matrix 
$$
\q_{11}=\begin{pmatrix} \q_{11}^{++} & \q_{11}^{+-} \\ \q_{11}^{-+} & \q_{11}^{--} \end{pmatrix}
$$
and 
$$
\rho_{12}(e)=\begin{pmatrix} \rho_{12}(e)^{++} & \rho_{12}(e)^{+-}\\ \rho_{12}(e)^{-+} & \rho_{12}(e)^{--} \end{pmatrix}
$$ 
for every $e\in E\cong \bR^{0,8}$. Equation (\ref{n2}) implies that
$$
\sqrt{2}(\q_{11}\circ e\cdot\circ\pi^{-}-e\cdot\circ\pi^{-}\circ\q_{11})=\rho_{12}(e)\circ\pi^{-}
$$
which is equivalent to
$$
\pi^{-}\circ\q_{11}|_{S^{+}}=\q_{11}^{-+}=0\quad,\quad \sqrt{2}(\q_{11}|_{S^{+}}\circ e\cdot-e\cdot\circ\pi^{-}\circ\q_{11}|_{S^{-}})=\rho_{12}(e)|_{S^{-}}
$$
\textit{i.e.} in matrix notation 
$$
\q_{11}=\begin{pmatrix} \q_{11}^{++} & \q_{11}^{+-} \\ 0 & \q_{11}^{--} \end{pmatrix}
\phantom{c},\phantom{c}\rho_{12}(e)=\begin{pmatrix} \rho_{12}(e)^{++} & \sqrt{2}(q_{11}^{++}\circ e\cdot-e\cdot\circ\q_{11}^{--}) \\ \rho_{12}(e)^{-+} & 0 \end{pmatrix}\qquad.
$$
Equation (\ref{n3}) implies that
$$
\sqrt{2}(\q_{22}\circ e\cdot\circ\pi^{+}-e\cdot\circ\pi^{+}\circ\q_{22})=\pi^{-}\circ\rho_{12}(e)
$$
which is equivalent to
$$
\pi^{+}\circ\q_{22}|_{S^{-}}\q_{22}^{+-}=0\quad,\quad \sqrt{2}(\q_{22}^{--}\circ e\cdot-e\cdot\circ\phantom{c}\q_{22}^{++})=\rho_{12}(e)^{-+}
$$
\textit{i.e.} in matrix notation
$$
\q_{22}=\begin{pmatrix} \q_{22}^{++} & 0 \\ \q_{22}^{-+} & \q_{22}^{--} \end{pmatrix}
$$
\be
\label{modulo}
\rho_{12}(e)=\begin{pmatrix} \rho_{12}(e)^{++} & \sqrt{2}(q_{11}^{++}\circ e\cdot-e\cdot\circ\q_{11}^{--}) \\ \sqrt{2}(\q_{22}^{--}\circ e\cdot-e\cdot\circ\phantom{c}\q_{22}^{++}) & 0 \end{pmatrix}
\ee
Equation (\ref{n1}) implies that (\ref{modulo}) coincides with
$$
\sqrt{2}\begin{pmatrix} e\cdot\circ\phantom{c}\q_{22}^{-+}-\q_{11}^{+-}\circ e\cdot\phantom{cccc} & \phantom{cccc}e\cdot\circ\phantom{c}\q_{22}^{--} -\q_{11}^{++}\circ e\cdot \\ e\cdot\circ\phantom{c}\q_{22}^{++}-\q_{11}^{--}\circ e\cdot \phantom{cccc} & \phantom{cccc}0 \end{pmatrix}\qquad.
$$
In particular, from the off-diagonal entries, we get that
$$
2e\cdot\circ\phantom{c}\q_{22}^{++}=(\q_{11}^{--}+\q_{22}^{--})\circ e\cdot\qquad,\qquad 2\q_{11}^{++}\circ e \cdot=e\cdot\circ (\q_{11}^{--}+\q_{22}^{--})\qquad.
$$
It follows that the elements $\q_{11}^{++},\q_{22}^{++}$ belong to $\End_{\bR}(S_{0,8}^{+})^{\so(0,8)}$ \textit{i.e.} they commute with the action of $\so(0,8)$ on the $\so(0,8)$-irreducible module of real type $S^{+}_{0,8}$. This implies that $\q_{11}^{++},\q_{22}^{++}$
are scalars satisfying
$$
2\q_{22}^{++}=2\q_{11}^{++}=\q_{11}^{--}+\q_{22}^{--}\in\bR\qquad.
$$
The only non-trivial diagonal entry of (\ref{modulo}) implies that 
$$
\rho_{12}(e)^{++}=\sqrt{2}(e\cdot\circ\phantom{c}\q_{22}^{-+}-\q_{11}^{+-}\circ e\cdot)\qquad.
$$ 
Summarizing we get that
$$\q_{11}=\begin{pmatrix} \q_{11}^{++} & \q_{11}^{+-} \\ 0 & \q_{11}^{--} \end{pmatrix}
\phantom{c},\phantom{c}\q_{22}=\begin{pmatrix} \q_{11}^{++} & 0 \\ \q_{22}^{-+} & \q_{22}^{--} \end{pmatrix}$$
\be
\label{suitsuit}
\rho_{12}(e)=\sqrt{2}\begin{pmatrix} e\cdot\circ\phantom{c}\q_{22}^{-+}-\q_{11}^{+-}\circ e\cdot & q_{11}^{++}\circ e\cdot-e\cdot\circ\q_{11}^{--} \\ \q_{22}^{--}\circ e\cdot-e\cdot\circ\phantom{c}\q_{11}^{++} & 0 \end{pmatrix}
\ee
where
$
2\q_{11}^{++}=\q_{11}^{--}+\q_{22}^{--}\in\bR
$. 
We solve equation (\ref{n4}). It involves the endomorphism 
$$\q_{12}=\begin{pmatrix} \q_{12}^{++} & \q_{12}^{+-} \\ \q_{12}^{-+} & \q_{11}^{--} \end{pmatrix}
$$
and implies the stated restrictions on the symmetric operator $B$. An explicit calculation shows that (\ref{n4}) is then equivalent to
\be
\label{3333}
\sqrt{2}(\q_{11}^{--}\q_{22}^{--}-(\q_{11}^{++})^{2})\circ e\cdot+h_{1}\q_{12}^{--}\circ e\cdot=\sqrt{2}Be|_{S^{+}}
\ee
$$
\sqrt{2}(\q_{11}^{+-}\q_{22}^{--}-\q_{11}^{+-}\q_{11}^{++})e\cdot+h_{1}q_{12}^{+-}e\cdot=
$$
\be
\label{3333c}
e\cdot\circ\sqrt{2}(\q_{11}^{++}\q_{22}^{-+}-\q_{11}^{--}\q_{22}^{-+})-e\cdot\circ\pc h_{1}q_{12}^{-+}
\ee
Evaluating equation (\ref{3333}) on $e_{i}$ for $1\leq i\leq 8$, we get that
\be
\label{èdef3}
\sqrt{2}(\q_{11}^{--}\q_{22}^{--}-(\q_{11}^{++})^{2})+h_{1}\q_{12}^{--}=\sqrt{2}b_{i}
\ee
Since the left-hand side of (\ref{èdef3}) does not depend on $1\leq i\leq n-1$, it follows that $B=\pm\Id$ (recall assumptions (\ref{prima}) on B). Both terms of equation (\ref{3333c}) are $\so(0,8)$-compatible linear maps between the two irreducible inequivalent $\so(0,8)$-modules $S_{0,8}^{\pm}$. This implies that
\be
\label{èdef}
\sqrt{2}(\q_{11}^{+-}\q_{22}^{--}-\q_{11}^{+-}\q_{11}^{++})+h_{1}q_{12}^{+-}=0
\ee
\be
\label{èdef2}
\sqrt{2}(\q_{11}^{++}\q_{22}^{-+}-\q_{11}^{--}\q_{22}^{-+})- h_{1}q_{12}^{-+}=0
\ee
Equations (\ref{èdef3}), (\ref{èdef}) and (\ref{èdef2}) define three linear maps $\q_{12}^{--}$, $\q_{12}^{+-}$ and $\q_{12}^{-+}$. To complete the proof of the theorem, we only need to check that there exist appropriate $\q_{11}^{++},\q_{11}^{+-},\q_{11}^{--},\q_{22}^{-+},\q_{22}^{--}$ such the action of $E$ on $S$
$$
\rho(e)=\begin{pmatrix} -h_{1}e\circ\pi^{-} & \rho_{12}(e) \\ 0 & +h_{1}e\circ\pi^{+}  \end{pmatrix}
$$
is associated with a suitable map (\ref{suitsuit}). The trival solution $\q_{11}=\q_{22}=0$ works.
\end{proof}
\begin{remark}\rm{
One can check that
\begin{enumerate}
\item[i)]
Theorem \ref{check} holds for any CW algebra with $\dim E\equiv 8$ (mod\pc $8$),
\item[ii)] Formula (\ref{formuloneoneone}) gives non-zero charge spin representation of a CW algebra $\gg_{b}$ whenever $b=\pm\Id$ and semi-spinors exist. 
\end{enumerate}
}\end{remark}
\section{Superization of the Cahen-Wallach algebra}
\label{dan2}
\setcounter{equation}{0}
In the following theorem the triple (\ref{triple2}) is defined up to elements of the kernel of the spin representation $\rho_{0,n-1}:\Cl_{0,n-1}\rightarrow\End_{\bR}(S_{0,n-1})$, if it exists.
\begin{theorem}
\label{zzrr}
Every zero charge superization with translational supersymmetry $$\gg=\gg_{\0}+\gg_{\ou}=\gg_{b}+S$$  of a $(n+1)$-dimensional $CW$ algebra $$\gg_{b}=\gh+\gm=E^{*}+(E+\bR\p+\bR\q)$$ is uniquely determined by a solution
\be
\label{triple2}
(c_{1},c_{2})\in(\Cl_{0,n-1})^{2}
\ee
of the {\bf quadratic Clifford equation} (\ref{QUAD}), by an element $c_{3}\in\Cl_{0,n-1}$ and by two bilinear forms 
\be
\label{twobil}
\p_{-+}\in\Bil_{\bR}(S_{0,n-1})^{\tau\sigma=-1}\quad,\quad \p_{++}\in S_{0,n-1}^{*}\vee S_{0,n-1}^{*}
\ee
which satisfy the {\bf compatibility conditions} ($s,t\in S_{0,n-1}$)
$$
\p_{-+}(c_{3}s,t)+\p_{-+}(c_{3}t,s)+\p_{++}(c_{2}s,t)+\p_{++}(s,c_{2}t)=0
$$
$$
\pc\pc\pc\p_{-+}(c_{1}s,t)+\p_{-+}(s,c_{2}t)=0\pc\pc\pc.
$$
The associated zero charge superization is given by
$$
\ad_{\mathrm{p}}|_{S}=\begin{pmatrix} 0 & 0 \\ 0 & 0 \end{pmatrix}\quad,\quad
\ad_{\q}|_{S}=\begin{pmatrix} c_{1} & c_{3} \\ 0 & c_{2} \end{pmatrix}
$$
$$
\phantom{cccccc}\ad_{e^{\flat}}|_{S}=\sqrt{2}\begin{pmatrix} 0 & Be \\ 0 & 0 \end{pmatrix}\phantom{cccc}\quad,\quad
\ad_{e}|_{S}=\sqrt{2}\begin{pmatrix} 0 & e\cdot\circ\phantom{c} c_{2}-c_{1}\circ\phantom{c}e\cdot\\ 0 & 0 \end{pmatrix}
$$
$$
[Q,Q]=\left(\p_{++}(Q_{+},Q_{+})+2\p_{-+}(Q_{-},Q_{+})\right)\p
$$
where $e\in E$ and 
$Q=\left(
\begin{array}{c}
Q_{-} \\
Q_{+} \\
\end{array}
\right)\in S=S_{-}+S_{+}=S_{0,n-1}+S_{0,n-1}\quad.
$
\end{theorem}
\begin{proof}
The theorem is a straightforward consequence of Theorem \ref{iiii}, Lemma \ref{bilinear} and Lemma \ref{MM}.
We only need to remark that the compatibility conditions are equivalent to $\q$-invariance for the bracket of two odd elements and that $\q$-invariance together with $E^{*}$-invariance imply $E$-invariance.
\end{proof}
\font\smallit = cmti8
{\smallit Acknowledgement}\/. {\small The author is grateful to the University of Florence and I.N.d.A.M. for financial support during his stay in Edinburgh and to the University of Edinburgh for hospitality. This paper is essentially one chapter of the Ph.D. thesis the author wrote under the supervision of D.\ V.\ Alekseevsky. The author would like to thank him and A.\ Spiro for their constant encouraging support.}


\bigskip
\bigskip
\font\smallsmc = cmcsc8
\font\smalltt = cmtt8
\font\smallit = cmti8
\hbox{\parindent=0pt\parskip=0pt
\vbox{\baselineskip 9.5 pt \hsize=3.1truein
\obeylines
{\smallsmc
Andrea Santi
Dept. of Mathematics
University of Florence
Viale Morgagni 67/a  
Florence 50134 
Italy
Phone +390554237111
Fax +390554222695
}\medskip
{\smallit E-mail}\/: {\smalltt santi@math.unifi.it}}}


\begin{thebibliography}{ACDS}
\bigskip
\bibitem[1]{AC} D.\ V.\ Alekseevsky and V.\ Cortes, {\it Classification of N-(Super)-Extended Poincar\'e Algebras and Bilinear Invariants of the Spinor Representation of Spin(p,q)}, Comm.Math.Phys. {\bf 183} (1997) 477 - 510.
\bibitem[2]{ACDP} D.\ V.\ Alekseevsky, V.\ Cortes, C.\ Devchand, A.\ Proeyen, {\it Polyvector Super-Poincar\'e Algebras},
Comm.Math.Phys. {\bf 253} (2005) 385 - 422.
\bibitem[3]{ACDS} D.\ V.\ Alekseevsky, V.\ Cortes, C.\ Devchand, U.\ Semmelmann, {\it Killing spinors are Killing vector fields in Riemannian supergeometry}, Journ.Geom.Phys. {\bf 26} (1998) 37 - 50.
\bibitem[4]{CKN} E.\ Chang-Young, H.\ Kim, H.\ Nakajima, {\it Noncommutative superspace and super Heisenberg group}, J.High Energy Phys. {\bf 4} (2008) 16 pp.
\bibitem[5]{CJS} E.\ Cremmer, B.\ Julia, J.\ Scherk, {\it Supergravity Theory in 11 dimensions}, Phys.Lett.B {\bf 76} (1978) 409 - 412.
\bibitem[6]{CW} M.\ Cahen and N.\ Wallach, {\it Lorentzian Symmetric Spaces}, Bull.Amer.Math.Soc. {\bf 76} (1970) 585 - 591.
\bibitem[7]{DFLV} R.\ D'Auria, S.\ Ferrara, M.\ A.\ Lledo, V.\ S.\ Varadarajan, {\it Spinor algebras}, J.Geom.Phys. {\bf 40} (2001) 101 - 129.
\bibitem[8]{FGH} S.\ Fernando, M.\ Gunaydin, S.\ Huyn, {\it Oscillator construction of spectra of pp-wave superalgebras in eleven dimensions}, Nucl.Phys.B {\bf 727} (2005) 421 - 460.
\bibitem[9]{FP} J.\ Figueroa-O'Farrill, G.\ Papadopoulos, {\it Homogeneous fluxes, branes and a maximally supersymemtric solution of M-theory}, J.High Energy Phys. {\bf 8} (2001) 26pp.
\bibitem[10]{H} F.\ R.\ Harvey, {\it Spinors and Calibrations}, Academic Press (1990).
\bibitem[11]{K} V.\ G.\ Kac, {\it Lie superalgebras}, Advances in Mathematics {\bf 26} (1977) 8 - 96.
\bibitem[12]{Kl1} F.\ Klinker, {\it Supersymmetric Killing structures}, Comm.Math.Phys. {\bf 255} (2005) 419 - 467.
\bibitem[13]{Kl2} F.\ Klinker, {\it SUSY structures on deformed supermanifolds}, Diff.Geom.Appl. {\bf 26} (2008) 566 - 582.
\bibitem[14]{Kt} B.\ Kostant, {\it Graded manifolds, graded Lie theory and prequantization}, Springer LNM {\bf 570} (1977) 177 - 306.
\bibitem[15]{Kz} J.\ L.\ Koszul, {\it Graded manifolds and graded Lie algebras}, Proceedings of the International Meeting on Geometry and Physics (Bologna), Pitagora (1982) 71 - 84.
\bibitem[16]{LM} H.\ B.\ Lawson and M.\ Michelson,
{\it Spin Geometry},
Princ.Univ.Press (1989).
\bibitem[17]{N} W.\ Nahm, {\it Supersymmetries and their representations}, Nucl.Phys.B {\bf 135} (1977) 149.
\bibitem[18]{OV} A.\ L.\ Onishchik, E.\ L.\ Vinberg, {\it Lie groups and algebraic groups}, Springer (1990). 
\bibitem[19]{S} A.\ Santi, {\it Invariant superconnections on homogeneous supermanifold and superizations of homogeneous manifold}, in preparation.
\bibitem[20]{Sc} M.\ Scheunert, {\it The theory of Lie superalgebras}, Lect.Notes in Math. {\bf 716} (1979).
\bibitem[21]{Z} F.\ Zhang,
{\it Quaternions and Matrices of Quaternions}, Linear Algebra and its Applic. {\bf 251}, (1997) 21 - 57.
\end{thebibliography}
\end{document}